\documentclass[a4paper,11pt]{article}

\usepackage[T1]{fontenc}
\usepackage[utf8]{inputenc}
\usepackage[english]{babel}
\usepackage{color}
\usepackage{amsmath,amssymb,amsthm}
\usepackage{ stmaryrd }
\usepackage{latexsym}
\usepackage{graphicx,epsfig}
\usepackage{float,enumerate,array}
\usepackage[left=25mm, right=25mm,
            top=30mm, bottom=30mm]{geometry}
\usepackage{pifont} 
\usepackage{dsfont}

\newtheorem{theo}{Theorem}
\newtheorem{prop}[theo]{Proposition}
\newtheorem{lem}[theo]{Lemma}

\newtheorem{remark}[theo]{Remark}

\newcommand{\bbR}{\mathbb{R}}

\newcommand{\set}[1]{\left\{#1\right\}}

\newcommand{\eps}{\varepsilon}

\renewcommand{\P}{\mathbb{P}}

\newcommand{\Z}{\mathbb{Z}}
\newcommand{\R}{\mathbb{R}}

\usepackage{hyperref,pifont}
\usepackage{todonotes}

\begin{document}

\title{Longest increasing paths with Lipschitz constraints}
\author{%
\textsc{Anne-Laure Basdevant}\footnote{Laboratoire Modal'X, UPL,  Univ. Paris Nanterre, France. email: anne.laure.basdevant@normalesup.org. Work partially supported by ANR PPPP, ANR Malin and Labex MME-DII.} \and \textsc{Lucas Gerin}\footnote{CMAP, Ecole Polytechnique, France. email: gerin@cmap.polytechnique.fr.  Work partially supported by ANR PPPP and ANR GRAAL.}}

\maketitle
\begin{abstract}
The Hammersley problem asks for the maximal number of points in a monotonous path through a Poisson point process. It is exactly solvable and notoriously known to belong to the KPZ universality class, with a cube-root scaling for the fluctuations.

Here we introduce and analyze a variant in which we impose a Lipschitz condition on paths. Thanks to a coupling with the classical Hammersley problem we observe that this variant is also exactly solvable. It allows us to derive first and second order asymptotics. It turns out that the cube-root scaling only holds for certain choices of the Lipschitz constants.
\end{abstract}
 
\noindent{\bf {\textsc MSC 2010 Classification}:} 60K35, 60F15.\\
\noindent{\bf Keywords:} combinatorial probability, last-passage percolation, longest increasing subsequences, longest increasing paths, Hammersley's process, cube-root fluctuations.

\section{Introduction}

The focus of the present paper is the Hammersley (or Ulam-Hammersley) problem. It asks for the asymptotic behaviour of the maximal number $\ell_n$ of points  that belong to a monotonous path through $n$ uniform points in a square (or more precisely a Poissonized version thereof, see details below).
It is among the few exactly solvable models which are known to lie in the \emph{Kardar-Parisi-Zhang (KPZ) university class}, another exactly solvable model is the corner-growth model (which is itself closely related to last-passage percolation).
The main feature of the KPZ class is a universal relation $\chi=2\xi -1$ between the \emph{fluctuation exponent} $\chi$ and the \emph{wandering exponent} $\xi$ (also called \emph{dynamic scaling exponent} or \emph{transversal exponent}). We refer the reader to \cite{CorwinKPZ} for an introduction to the KPZ class.

For the Hammersley problem, critical exponents have been shown to exist and we have $\chi=1/3$ (\cite{BDK}) and $\xi=2/3$ (\cite{JohanssonTransversal}). The proof of $\chi=1/3$ requires \emph{hard} analysis and the deep connection between $\ell_n$ and the shape of random Young tableaux (see \cite{Romik}). So-called \emph{soft} arguments were gathered since (\cite{AldousDiaconis,CatorGroeneboom,CatorGroeneboom2}) and lead to a weaker form of $\chi=1/3$ but they are of course still non-trivial.

Various authors have studied geometric properties of maximizing paths in the Hammersley problem and last-passage percolation.
In a non-exhaustive way one can mention: convexity constraint \cite{Barany}, H\"older constraint \cite{BergerTorri,BergerTorri2}, gaps constraint \cite{Gaps}, localization constraint \cite{DeyJosephPeled}, area below the path \cite{Basu}, modulus on continuity \cite{Hammond}.

In particular some of the above works have studied the robustness of the membership to KPZ under such modifications of the constraints.  Our goal here is mainly to illustrate the (non)-robustness of $\chi=1/3$ for the Hammersley process when we impose a Lipschitz condition on paths. It turns out that this problem is easily seen to be exactly solvable, thanks to a coupling with the classical Hammersley problem.

We now formally introduce our model.
Let $\Xi$ be a homogeneous Poisson process with intensity one in the quarter-plane $[0,+\infty)^2$.
Hammersley \cite{HammersleyHistorique} studied the problem of the maximal number of points $L(x,t)$ in an increasing path of points of $\Xi \cap \left([0,x]\times[0,t]\right)$.  He used subadditivity to prove the existence of a constant $\pi/2 \leq c \leq e$ such that $L(t,t)/t \to c$ in probability and conjectured $c=2$. This result
 was first obtained by Ver{\v{s}}ik-Kerov \cite{VersikKerov}. We give here a formulation due to Aldous-Diaconis \cite{AldousDiaconis} which provided a more probabilistic proof by exploiting the geometric construction of Hammersley:
\begin{prop}[\cite{AldousDiaconis}, Th.5 (a)]\label{AldousDiaconis}
Let $a,b>0$. Then
\begin{equation}\label{eq:LimiteL_continu}
f(a,b):=\lim_{t\to +\infty}\frac{L(at,bt)}{t}= 2\sqrt{ab}.
\end{equation}
The convergence holds a.s. and in $L^1$.
\label{Prop:AldousDiaconis}
\end{prop}

Let $0<\alpha<\beta$ be fixed parameters, we introduce the partial order relation $\stackrel{\alpha,\beta}{\preccurlyeq}$ in $[0,+\infty)^2$ defined by
$$
(x,y)\stackrel{\alpha,\beta}{\preccurlyeq}(x',y') \Leftrightarrow 
\begin{cases}
&x \leq x',\\
&\alpha \leq \frac{y'-y}{x'-x}\leq \beta.
\end{cases}
$$
In the extremal case $\alpha=0$, $\beta=+\infty$ then $\stackrel{0,+\infty}{\preccurlyeq}$ is the classical partial order on  $[0,+\infty)^2$ defined by $\{x \leq x',y\leq y'\}$.
For any domain $D \subset [0,+\infty)^2$, we consider the random variable $L^{\bf \alpha,\beta}(D)$ given by the length of the longest chain in $\Xi \cap D$  with respect to $\stackrel{\alpha,\beta}{\preccurlyeq}$:
\begin{multline*}
L^{\alpha,\beta}(D)=
\max\bigg\{L; \hbox{ there are }(u_1,v_1),\dots,(u_L,v_L)\in \Xi \cap D \text{ such that }\\
(u_1,v_1) \stackrel{\alpha,\beta}{\preccurlyeq} (u_2,v_2)  \stackrel{\alpha,\beta}{\preccurlyeq} \dots  \stackrel{\alpha,\beta}{\preccurlyeq} (u_L,v_L)
\bigg\}.
\end{multline*}
When $D$ is a rectangle of the form $[0,x]\times[0,t]$, we just write $L^{\alpha,\beta}(x,t)$ instead of $L^{\alpha,\beta}(D)$.
In the case $\alpha=0$, $\beta=+\infty$, we recover the usual quantity $L^{0, +\infty}(x,t):=L(x,t)$ studied by Hammersley.

\begin{figure}
    \begin{minipage}[c]{.46\linewidth}
        \centering
        \includegraphics[width=7cm]{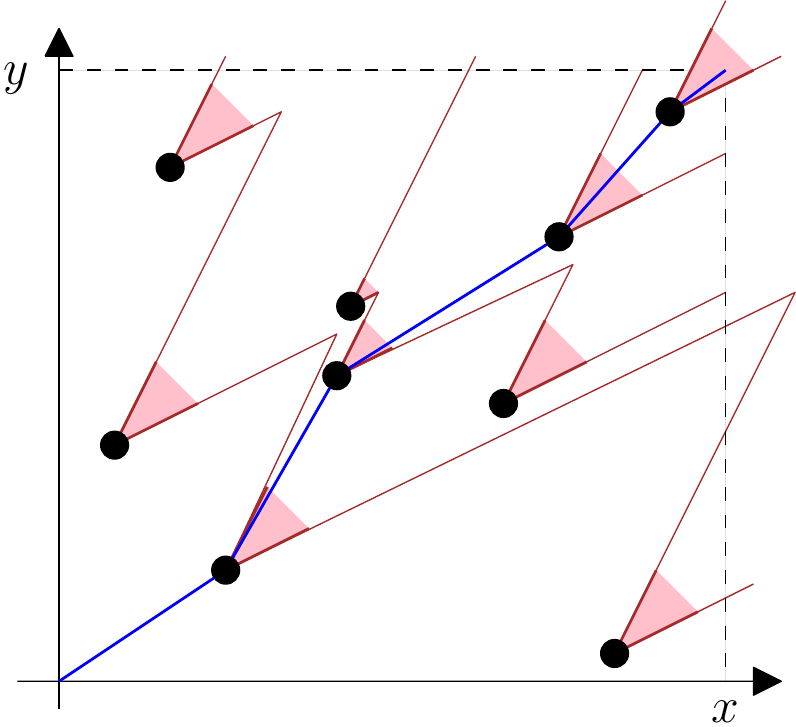}
        \caption{An example where $L^{\alpha,\beta}(x,y)=4$. Every point $(u,v)$ of $\Xi$ (represented by $\bullet$'s) is drawn with its set of majorants for $\stackrel{\alpha,\beta}{\preccurlyeq}$. One maximizing path is drawn in blue.}
\label{fig:DefinitionModele}
    \end{minipage}
    \hfill%
    \begin{minipage}[c]{.46\linewidth}
        \centering
        \includegraphics[width=7cm]{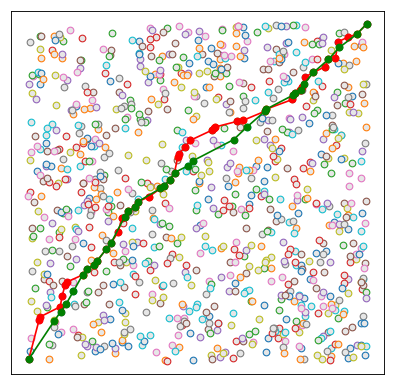}
        \caption{The maximazing paths through $n=1000$ points (red: $\alpha=1/6,\beta=6$, green: $\alpha=1/2,\beta=2$.)}
    \end{minipage}
\end{figure}

The first  result of this paper is that we can also explicitly compute the limit of $L^{\alpha,\beta}(at,bt)/t$.
\begin{theo}[Limiting shape]\label{MaxiTheoreme}
For every $a,b>0$, every $\beta>\alpha\ge 0$, we have the following almost-sure limit: 
$$
f^{\alpha,\beta}(a,b):=\lim_{t\to \infty}\frac{L^{\alpha,\beta}(at,bt)}{t}=
\begin{cases}
a\sqrt{\beta-\alpha}
&\text{ if }\frac{b}{a} > \frac{\alpha+\beta}{2},\\
\displaystyle{\frac{2\sqrt{(\beta a-b)( b-a\alpha)}}{\sqrt{\beta -\alpha}}} &\text{ if }
 \frac{\alpha+\beta}{2} \geq \frac{b}{a} \geq  \frac{2}{1/\beta+1/\alpha} ,\\
b\sqrt{\frac{1}{\alpha}-\frac{1}{\beta}}
 &\text{ if }\frac{b}{a} < \frac{2}{1/\beta+1/\alpha}.
\end{cases}
$$
\end{theo}
\noindent \emph{(Basic calculus shows that we always have $\frac{\alpha+\beta}{2} \geq  \frac{2}{1/\beta+1/\alpha}$.)}

In the sequel, the condition $\frac{\alpha+\beta}{2} \geq \frac{b}{a} \geq  \frac{2}{1/\beta+1/\alpha}$ will be referred to as the \emph{central case}.

A key consequence of Theorem \ref{MaxiTheoreme} is that the limiting shape is strictly convex exactly within the central case.
The proof of Theorem \ref{MaxiTheoreme} is only simple calculus once we notice a coupling with the classical Hammersley problem (see Proposition \ref{Prop:Coupling2} below).
What is more appealing is that we obtain non-trivial dichotomies for the shapes of paths and the fluctuations.
The following result states that in the central case all the maximizing paths  are localized along the main diagonal $\set{y=xb/a}$ (asymptotically with high probability) while in the two non-central cases, optimal paths have slope respectively $\frac{\alpha+\beta}{2}$ or $\frac{2}{1/\beta+1/\alpha}$. In order to state the result, we set for $R >0$ and a \emph{slope} $\theta \in \bbR$
$$
D_{R,\theta} =\set{(x,y)\in \R_+^2\;;\; |y-\theta x|\leq R}.
$$

\begin{theo}[Localization of paths]\label{th:Localization}
\renewcommand {\theenumi}{\emph{\roman{enumi})}}
\ 
\begin{enumerate}
\item If $\frac{\alpha+\beta}{2} \geq \frac{b}{a} \geq  \frac{2}{1/\beta+1/\alpha}$ then for every $\delta >0$
$$
\mathbb{P}
\begin{pmatrix}
\text{ There is an optimal path for }L^{\alpha,\beta}(at,bt)\\
\text{ not wholly contained in }D_{t\delta,\tfrac{b}{a}}
\end{pmatrix}
 \stackrel{t\to +\infty}{\to} 0.
$$
\item If $\frac{b}{a} > \frac{\alpha+\beta}{2}$ then for every $\delta >0$
$$
\mathbb{P}
\begin{pmatrix}
\text{ There is an optimal path for }L^{\alpha,\beta}(at,bt)\\
\text{ not wholly contained in }(0,c) + D_{t\delta,\tfrac{\alpha+\beta}{2}}\text{ for some }c\geq 0
\end{pmatrix}
\stackrel{t\to +\infty}{\to} 0.
$$
\item If $\frac{b}{a} <  \frac{2}{1/\beta+1/\alpha}$ then for every $\delta >0$
$$
\mathbb{P}
\begin{pmatrix}
\text{ There is an optimal path for }L^{\alpha,\beta}(at,bt)\\
\text{ not wholly contained in }(c,0) + D_{t\delta,\tfrac{2}{1/\beta+1/\alpha}}\text{ for some }c\geq 0
\end{pmatrix}
 \stackrel{t\to +\infty}{\to} 0.
$$
\end{enumerate}
\end{theo}
(Three cases of Theorem \ref{th:Localization} are illustrated in the rightmost column of table p.\pageref{Tableau}.)

It is a folklore result that in the classical Hammersley problem (\emph{i.e.} item (i) above with $\alpha=0, \beta=+\infty$) asymptotically all the optimizing paths concentrate on the diagonal $y=xb/a$. We cannot date back the first appearance of this result but it can be seen as a consequence of a more general result by Deuschel-Zeitouni \cite[Th.2 (ii)]{DeuschelZeitouni}. 
In the case $b/a >\beta$ the diagonal is not an admissible path so it is obvious that optimizing paths will not concentrate on it. However in the intermediate case $\beta> b/a > \frac{\alpha+\beta}{2}$ the diagonal is admissible but Theorem \ref{th:Localization} shows that paths close to the diagonal are not optimal.

\begin{figure}
\begin{center}
\includegraphics[width=7cm]{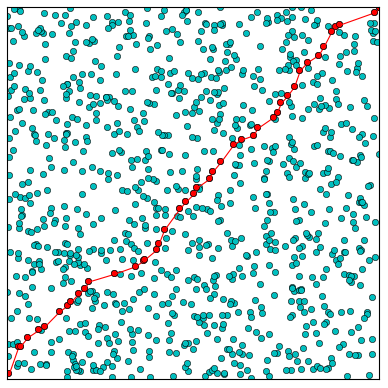}
\ \hspace{1cm}\ 
\includegraphics[width=7cm]{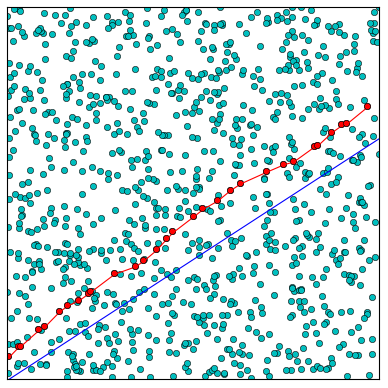}\\

\caption{Illustration of Theorem \ref{th:Localization}. Left: A maximizing path through $n=1000$ points with $\alpha=0.2,\beta=3$ (central case). Right: A maximizing path with $\alpha=0.2,\beta=1.1$ (non-central case). In blue: a straight line with slope $(\alpha+\beta)/2$.}
\end{center}
\end{figure}

We now discuss the fluctuations of $L^{\alpha,\beta}(at,bt)$. In 1999, Baik-Deift-Johansson have identified the fluctuations of $L(at,bt)$:

\begin{prop}[Cube root fluctuations for the Hammersley problem (\cite{BDK} Th.1.1)]\label{Prop:BDJ}\ 
For every $a,b>0$, we have the convergence in law 
\begin{equation*}
\frac{L(at,bt) -2\sqrt{ab}t}{(\sqrt{ab}t)^{1/3}}
\stackrel{t\to +\infty}{\to} \mathrm{TW},
\end{equation*}
where $\mathrm{TW}$ is the  standard  Tracy-Widom distribution.
\end{prop}

We show that the same asymptotic behaviour is true within the central case. Contrary to Theorems \ref{MaxiTheoreme} and \ref{th:Localization} we are not able to characterize the behaviour on the boundaries  $b/a=\tfrac{\alpha+\beta}{2}$ and $b/a= \tfrac{2}{1/\beta+1/\alpha}$.

\begin{theo}[Cube root fluctuations for $L^{\alpha,\beta}$]\label{Th:FluctuLipschitz}
\renewcommand {\theenumi}{\emph{\roman{enumi})}}
\ 
\begin{enumerate}
\item If $\frac{\alpha+\beta}{2} > \frac{b}{a}> \frac{2}{1/\beta+1/\alpha}$, we have the convergence in law 
$$
\frac{L^{\alpha,\beta}(at,bt)-t f^{\alpha,\beta}(a,b)}
{\sigma_{a,b,\beta,\alpha}^{1/3} \times t^{1/3}}
\stackrel{t\to +\infty}{\to} \mathrm{TW},
$$
where $\mathrm{TW}$ is the standard  Tracy-Widom distribution and
$$
\sigma_{a,b,\beta,\alpha}= \frac{1}{2}f^{\alpha,\beta}=\frac{\sqrt{(\beta a-b)( b-a\alpha)}}{\sqrt{\beta -\alpha}}.
$$
\item If $\frac{b}{a} > \frac{\beta+\alpha}{2}$ or  $\frac{b}{a} < \frac{2}{1/\beta+1/\alpha}$, then
\begin{align*}
\frac{L^{\alpha,\beta}(at,bt)-t f^{\alpha,\beta}(a,b)}
{t^{1/3}} &\stackrel{t\to +\infty}{\to} +\infty \text{ in prob.,}\\
\frac{L^{\alpha,\beta}(at,bt)-t f^{\alpha,\beta}(a,b)}
{t^{1/3+\eps}} &\stackrel{t\to +\infty}{\to} 0 \text{ in prob. for every $\eps$.}
\end{align*}
\end{enumerate}
\end{theo}

\section{Proof of the limiting shape}

As mentioned earlier, the proof of Theorem \ref{MaxiTheoreme} follows from a coupling argument between $L^{\alpha,\beta }$ and $L^{0,+\infty}$. Thanks to this coupling we can split the proof of  Theorem  \ref{MaxiTheoreme} in two part: a probabilistic part (Proposition \ref{Prop:OptiLosange}) and simple lemma of optimization (Lemma \ref{Lem:geometrique_rho}).
Proofs of Proposition \ref{Prop:OptiLosange} and Lemma \ref{Lem:geometrique_rho} rely on quite usual and simple arguments but we detail all the computations as they will be needed in the (more involved) proofs of Theorems \ref{th:Localization} and \ref{Th:FluctuLipschitz}.

\begin{prop}[Coupling with the classical Hammersley problem]\label{Prop:Coupling2}
For every $x,y\ge 0$ and $\beta>\alpha\ge 0$ we have the identity
$$
L^{\alpha,\beta }([0,x]\times [0, y]) \stackrel{\text{(d)}}{=} L\left(\phi([0,x]\times[0,y])\right),
$$
where $\phi(x,y)$ is defined by
\begin{equation}\label{eq:DefiPhi2}
\phi(x,y)=\frac{\beta^{1/4}\alpha^{1/4}}{\sqrt{\beta-\alpha}}
\left( \beta^{1/2}x-\beta^{-1/2}y, -\alpha^{1/2}x+\alpha^{-1/2}y \right).
\end{equation}
\end{prop}
In fact, on can check that we have more generally the  equality in distribution between the two processes $\left(L^{\alpha,\beta }([0,x]\times [0, y]), x>0,y>0) \right)$ and $\left(L(\phi([0,x]\times[0,y])),x>0,y>0 \right)$. However, since we will just use this equality for fixed $x$ and $y$, we will not provide a proof here.

\begin{figure}
\begin{center}
\includegraphics[width=8cm]{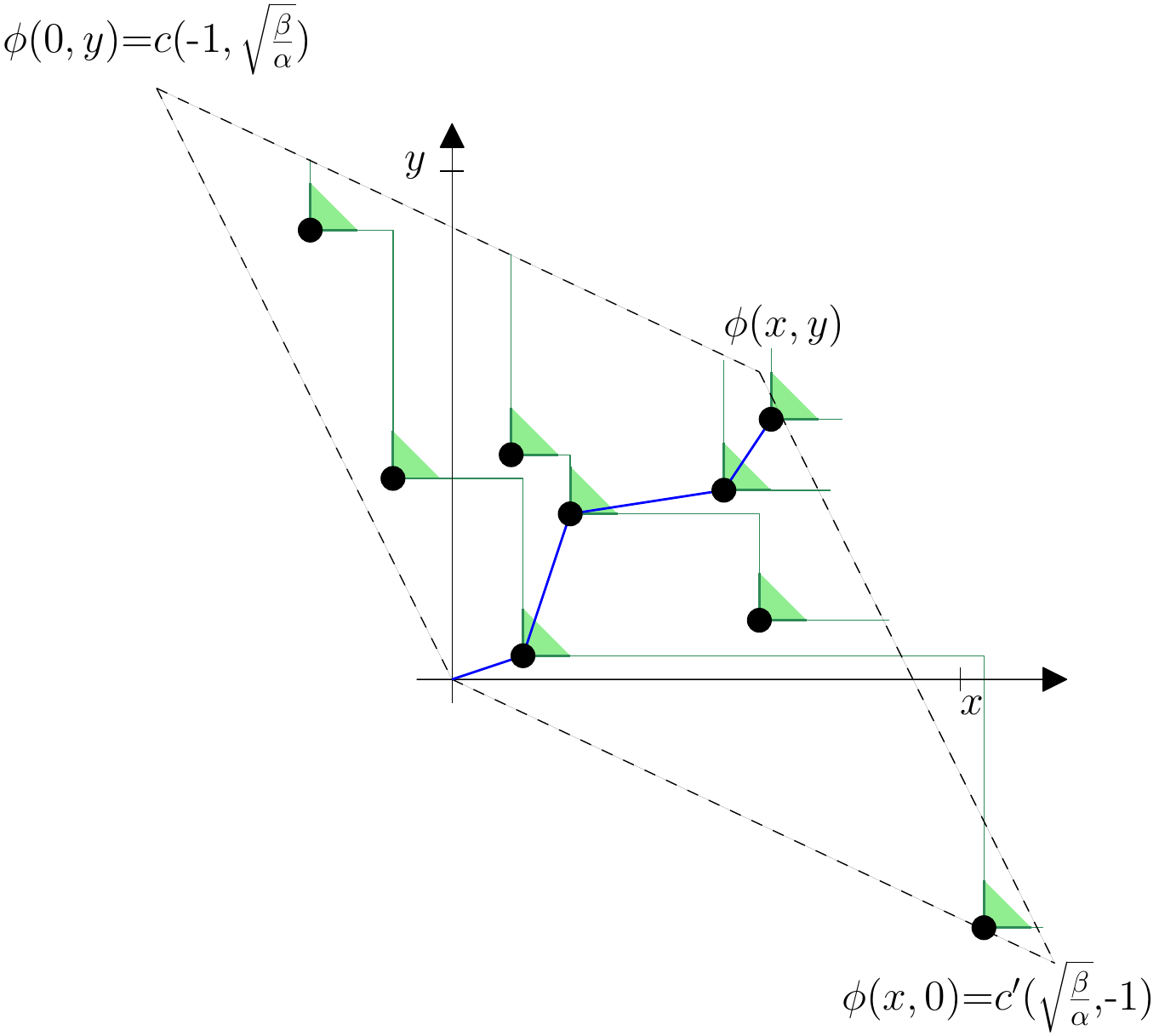}
\end{center}
\caption{The picture of Fig.\ref{fig:DefinitionModele} under map $\phi$. In particular, the image of the maximizing path for $\stackrel{\alpha,\beta}{\preccurlyeq}$ is mapped onto one maximal path in the parallelogram for the classical  order $\preccurlyeq$.}
\label{Fig:CouplageLosange}
\end{figure}

\begin{proof}[Proof of Proposition \ref{Prop:Coupling2}]
The linear function $\phi$ has determinant $1$ and is constructed so that
$$
(x,y)  \stackrel{\alpha,\beta}{\preccurlyeq} (x',y') \Leftrightarrow \phi(x,y)  \preccurlyeq \phi(x',y').
$$
Namely, straight lines $\{y=\alpha x\}$ (resp.  $\{y=\beta x\}$ ) are mapped onto horizontal (resp. vertical) lines.
Therefore the maximizing paths for the classical order $\preccurlyeq$ in $\phi([0,x]\times[0,y])$ are exactly the images by $\phi$ of the maximizing paths in $[0,x]\times[0,y]$ for the order $\stackrel{\alpha,\beta}{\preccurlyeq}$.

Besides, the fact that $\phi$ preserves areas implies that the image of $\Xi$ by $\phi$ is still a Poisson process with intensity one.
\end{proof}

We now solve the classical Hammersley problem in a parallelogram:
\begin{prop}\label{Prop:OptiLosange} Let $p,q,r,s>0$ and let $P_t,Q_t,R_t,S_t$ the points of the plane of respective coordinates $(-pt,qt)$, $(0,0)$,  $(rt,-st)$, $(rt-pt,qt-st)$. Let
$\mathcal{P}_t$ be the parallelogram  $P_t Q_t R_t S_t$. 
Then
\begin{equation}\label{eq:OptiLosange}
 \lim_{t\to \infty} \frac{L({\mathcal{P}_t})}{t} =
 \sup\{2\sqrt{{\sf Area}(\mathcal{R})}, \text{ such that the rectangle }\mathcal{R}:=[u,u']\times [v,v'] \subset \mathcal{P}_1\},
\end{equation}
where ${\sf Area}(\mathcal{R})$ denotes the area of $\mathcal{R}$.
\end{prop}

\begin{proof} 
For any $u,u',v,v'$ such that  $[u,u']\times [v,v'] \subset \mathcal{P}_1$, we have 
\begin{equation}\label{eq:MinoRectangleInclus}
L(\mathcal{P}_t)\ge L([ut,u't]\times [vt,v't]).
\end{equation}
By Proposition \ref{Prop:AldousDiaconis} we have that
\begin{equation*}
 \lim_{t\to \infty} \frac{ L([ut,u't]\times [vt,v't])}{t} = 2\sqrt{(u'-u)(v'-v)}=2\sqrt{{\sf Area}(\mathcal{R})}.
\end{equation*}
Thus we obtain the lower bound of \eqref{eq:OptiLosange}.

Let us now prove the upper bound. We discretize the boundary of $\mathcal{P}_t$ as follows. Let $(x_i,y_i)_{1\le i \le t^3}$ (resp. $(x_i,y'_i)$) be points of the plane such that
\begin{itemize}
\item For all $i$, $(x_i,y_i)$ lies on the segment $[P_t,Q_t]$ or on $[Q_t,R_t]$ (resp. $(x_i,y'_i)$ lies on the segment $[P_t,S_t]$ or on $[S_t,R_t]$)
\item $x_1=-pt$ and $x_{i+1}-x_i=(p+r)/t^2$. Thus $x_{t^3}=rt$.
\end{itemize}

\begin{center}
\includegraphics[width=9cm]{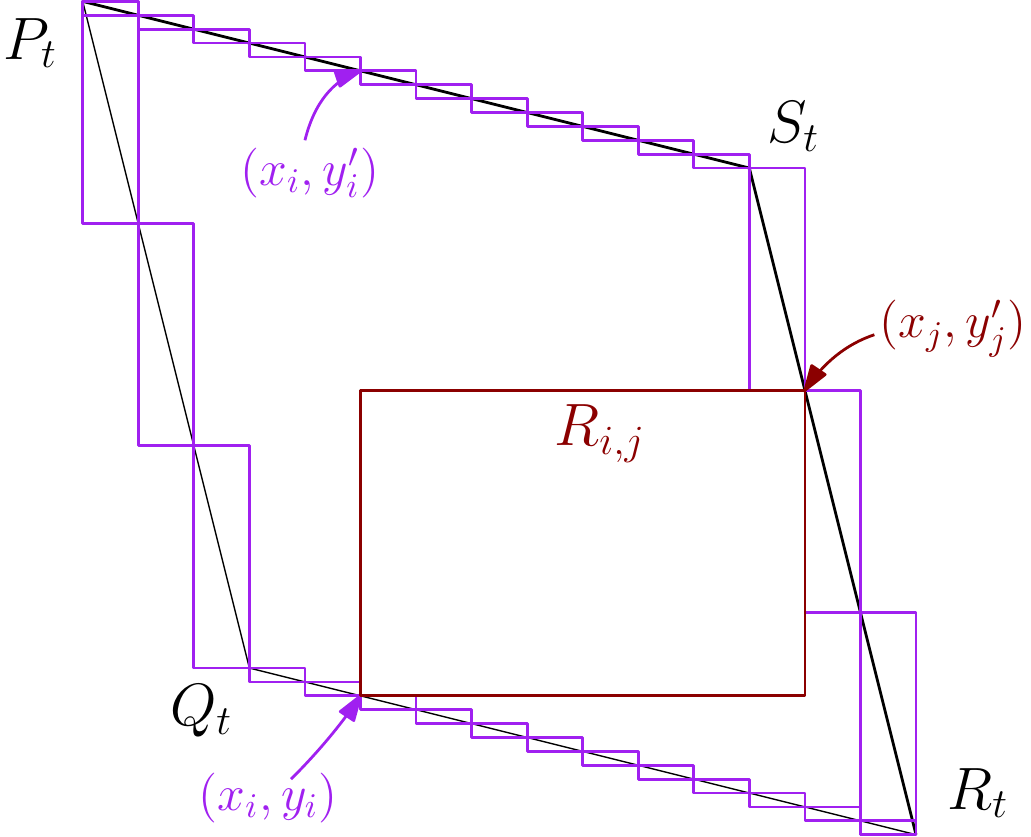}
\end{center}

Let $\mathcal{I}=\{(i,j),  i<j \mbox{ and } y_i\le y'_j\}$.
For $(i,j)\in \mathcal{I}$, denote $R_{i,j}$ the rectangle $[x_i,x_j]\times[y_i,y'_j]$ 
and $\mathcal{Q}_t=\bigcup_{(i,j)\in\mathcal{I}} R_{i,j}$.
Let us notice that $\mathcal{P}_t\setminus \mathcal{Q}_t$ is composed of $2t^3$ right triangles of width $(p+r)/t^2$ and of height bounded by $C/t^2$ for some constant $C$ depending of $p,q,r,s$. Thus the area of $\mathcal{P}_t\setminus \mathcal{Q}_t$ is of order $C/t$.
In particular 
$$\lim_{t\to \infty}\P(\mbox{ no point of $\Xi$ lies in } \mathcal{P}_t\setminus \mathcal{Q}_t)=1.$$
Thus 
$$\lim_{t\to \infty}\P(L(\mathcal{P}_t)= L(\mathcal{Q}_t))=1.$$
Due to the particular form of $\mathcal{Q}_t$, any increasing path included in $\mathcal{Q}_t$  is necessarily included in a rectangle $R_{i,j}$.
Thus
$$L(\mathcal{Q}_t)=\max\{L(R_{i,j}), (i,j)\in \mathcal{I}\}.$$
Hence for $\gamma >0$ 
\begin{eqnarray}
\P(L(\mathcal{Q}_t)\ge \gamma  t)&=& \P(\exists (i,j)\in \mathcal{I}, \; L(R_{i,j})\ge \gamma  t)\notag \\
&\le & \sum_{(i,j)\in \mathcal{I}} \P(L(R_{i,j})\ge \gamma  t) \notag \\
&\le & t^6 \max_{(i,j)\in \mathcal{I}} \P(L(R_{i,j})\ge \gamma  t) \label{eq:Majot6}
\end{eqnarray}
We now fix $\gamma >\sup\{2\sqrt{{\sf Area}(\mathcal{R})}, \text{ such that the rectangle }\mathcal{R}:=[u,u']\times [v,v'] \subset \mathcal{P}_1\}$.
In order to conclude we claim that there exists some constant $C:=C(\gamma)$, such that, for any $(i,j)\in \mathcal{I}$ and $t\geq 0$,
\begin{equation}\label{eq:Majot7}
\P(L(R_{i,j})\ge \gamma t)\le \frac{C}{t^7}.
\end{equation}
Indeed, this inequality is a direct consequence of \cite[Th.1.2]{BDK} (see also Proposition \ref{Prop:Flu_Ln_continu} below).
This concludes the proof of the upper bound.
\end{proof}

Hence, to obtain Theorem \ref{MaxiTheoreme}, it remains to determine the rectangle of maximal area in the parallelogram $\phi([0,a]\times [0,b])$.
The proof is easy but computations are easier to follow with a look at Figure  \ref{Fig:Table}. 

\begin{lem}\label{Lem:geometrique_rho} Let $\mu>1$, $c,c'>0$ and denote $P,Q,R,S$ the points of the plane of respective coordinates $(-c,c\mu)$, $(0,0)$,  $(c'\mu,-c')$, $(\sigma,\rho)$ with $\sigma:=-c+c'\mu$ and $\rho:=-c'+c\mu$ so that
the quadrilateral $\mathcal{P}:PQRS$ is a parallelogram. 
Then
\begin{equation}\label{Eq:geometrie}
 \max\{{\sf Area}(\mathcal{R}) \text{ such that  }\mathcal{R}:=[u,u']\times [v,v'] \subset \mathcal{P}\}=\begin{cases}
\displaystyle{\frac{(\rho\mu+\sigma)^2}{4\mu}}
 &\text{ if } \rho <\mu^{-1} \sigma\\
\displaystyle{\frac{(\sigma\mu+\rho)^2}{4\mu}}
&\text{ if }  \rho >\mu \sigma,\\
\displaystyle{\sigma\rho} &\text{ otherwise}.
\end{cases}
\end{equation}
More precisely,
\begin{itemize}
\item In the first case $\rho <\mu^{-1} \sigma$ a maximizing rectangle is given by $\mathcal{R}=[0,\mu\xi]\times [0,\xi]$ where $\xi=\frac{\rho\mu +\sigma}{2\mu}$. Every translation of $\mathcal{R}$ by a vector $u(\mu,-1)$ for some $u\in[0,(\sigma-\mu\rho)/2\mu]$ is also maximizing (a similar result holds in the second case).
\item In the third case $\mu^{-1}\sigma<\rho <\mu\sigma$ there is a unique rectangle $\mathcal{R}$ with maximal area and $\mathcal{R}=[0,\sigma]\times[0,\rho]$.
\end{itemize}

\end{lem}

\begin{figure}[h!]
\begin{center}
\begin{tabular}{|>{\centering}m{28mm} |>{\centering}m{50mm} |>{\centering}m{25mm} | m{50mm} |}
\hline
\multicolumn{2}{|c|}{After transformation $\phi$} & \multicolumn{2}{c|}{Before transformation $\phi$}\\
\hline
 & Maximizing rectangle(s) (drawn in red) & & Shape of maximizing path(s)\\
\hline
$\rho/\sigma <\mu^{-1}$  & \includegraphics[width=45mm]{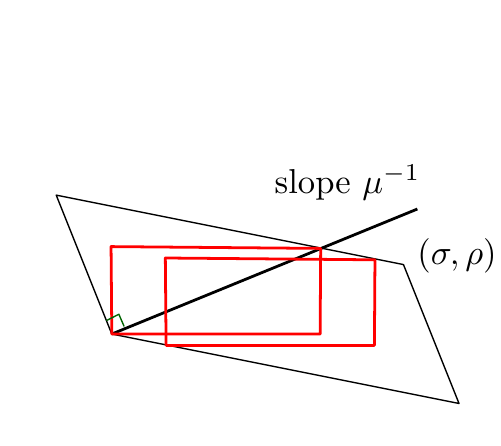} & $\frac{b}{a} <  \frac{2}{1/\beta+1/\alpha}$ &  \includegraphics[width=45mm]{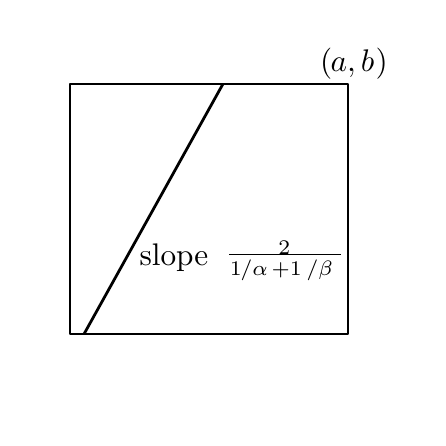} \\
\hline
$\rho/\sigma >\mu$ & \includegraphics[width=45mm]{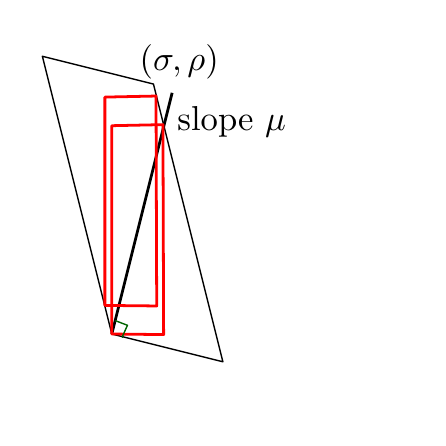} & $\frac{b}{a} >  \frac{\alpha+\beta}{2}$ &  \includegraphics[width=45mm]{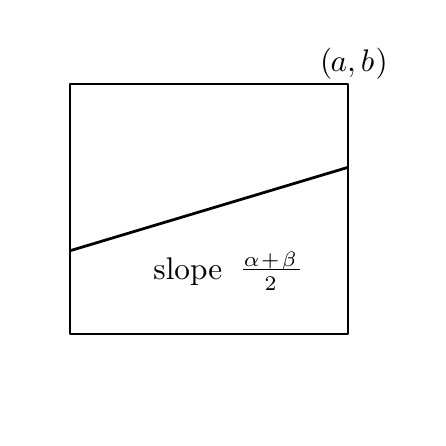} \label{Tableau} \\
\hline
$\mu^{-1}\leq \rho/\sigma \leq \mu$ & \includegraphics[width=45mm]{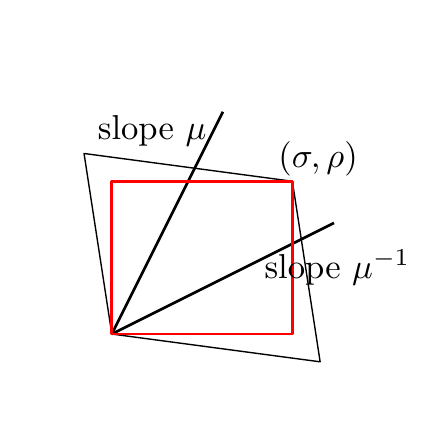} & central case &  \includegraphics[width=45mm]{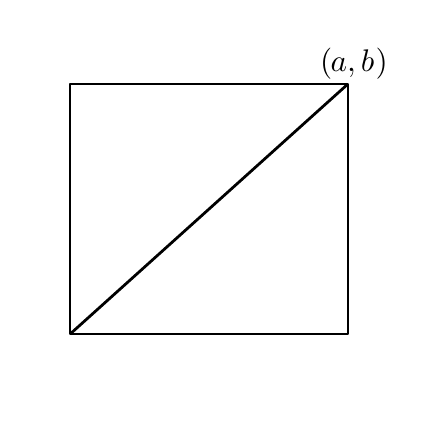} \\
\hline
\end{tabular}
\end{center}
\caption{Main notation and qualitative results for the three cases  of Theorem \ref{MaxiTheoreme} and Lemma \ref{Lem:geometrique_rho}}
\label{Fig:Table}
\end{figure}

\begin{proof}[Proof of Lemma \ref{Lem:geometrique_rho}] 
Set 
$$M:= \max\{{\sf Area}(\mathcal{R}) \text{ such that the rectangle }\mathcal{R}:=[u,u']\times [v,v'] \subset \mathcal{P}\}.$$
 We are going to identify a rectangle $\mathcal{R}:=[u,u']\times [v,v']:=ABCD$ inside $\mathcal{P}$ with the largest area. Clearly, such rectangle must have two opposite vertices on sides of $\mathcal{P}$, for instance $A$ and $C$, with $A$ on $[PQ]$ or $[QR]$ and $C$ on $[PS]$ or $[SR]$. With this convention, note first that if $ABCD\subset \mathcal{P}$, then the image by a translation of  vector $\overrightarrow{AQ}$ of this rectangle is  a rectangle inside  $\mathcal{P}$ with a vertex at $(0,0)$. Thus, we get 
 $$ M=
\max\{{\sf Area}(\mathcal{R}) \text{ such that the rectangle }\mathcal{R}:=[0,u]\times [0,v] \subset \mathcal{P}\}.$$
 
Assume first that  $\sigma,\rho >0$, so that the vertex $S$ is in the upper right quarter plane. Then, there are two  possible cases (drawn in red and blue in the picture below):
\begin{itemize}
\item Either $C\in [PS]$ (drawn in blue)  
\item  Or  $C\in [SR]$ (drawn in red)
\end{itemize}
\begin{center}
\includegraphics[width=12cm]{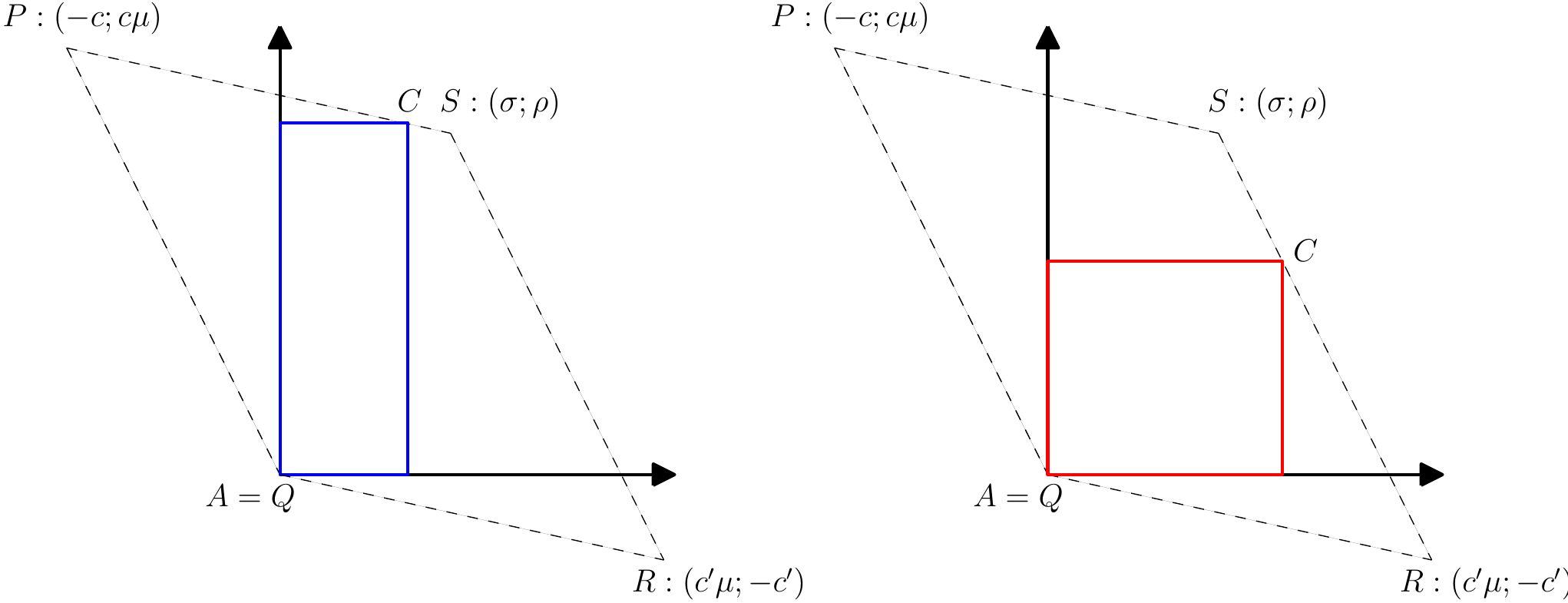}\\
\end{center}
We now distinguish the two cases. If $C\in [PS]$, then $C:=(-c+xc'\mu,c\mu-xc')$ for some $x\in[c/(c'\mu),1]$ since $C$ must be in $(\mathbb{R_+})^2$.
Thus
$${\sf Area}(\mathcal R)=(-c+xc'\mu)\times (c\mu-xc')\quad
\mbox{ which is maximal for }\quad x=\min\left(\frac{c}{2c'}(\mu+\frac{1}{\mu}),1\right).$$
If $C\in [SR]$, then $C:=(c'\mu-xc,-c'+xc\mu)$ for some $x\in[c'/(c\mu)),1]$ since $C$ must be in $(\mathbb{R_+})^2$.
Thus
$${\sf Area}(\mathcal R)=(c'\mu-xc)\times(-c'+xc\mu) \quad
\mbox{ which is maximal for }\quad x=\min\left(\frac{c'}{2c}(\mu+\frac{1}{\mu}),1\right).$$
Note that, assuming $\rho,\sigma>0$,
$$\frac{\rho}{\sigma}>\mu  \Longleftrightarrow c\mu-c'>\mu(-c+c'\mu) \Longleftrightarrow  1-\frac{c'}{c\mu}> -1+\frac{c'\mu}{c} \Longleftrightarrow  \frac{c'}{2c}(\mu+\frac{1}{\mu})<1$$
Symmetrically, we have  
$$\frac{\rho}{\sigma}<\mu^{-1}  \Longleftrightarrow  \frac{c}{2c'}(\mu+\frac{1}{\mu})<1$$

Hence, if $\mu^{-1}<\rho/\sigma<\mu$, the maximum are in both cases obtained for $x=1$, \emph{i.e.} $C=S$. On the other hand, if this condition does not hold, the maximum is obtained for some $C\neq S$, thus the image of $ABCD$ by a translation of vector  $\overrightarrow{CS}$ yields another rectangle inside $\mathcal{P}$ with the same area. Then, simple calculus yields the claimed expression for $M$.

Assume now that the vertex $S$ is not in the upper right quarter plane, for example we have $\sigma>0$ but $\rho<0$ as in the picture below.
\begin{center}
\includegraphics[width=7cm]{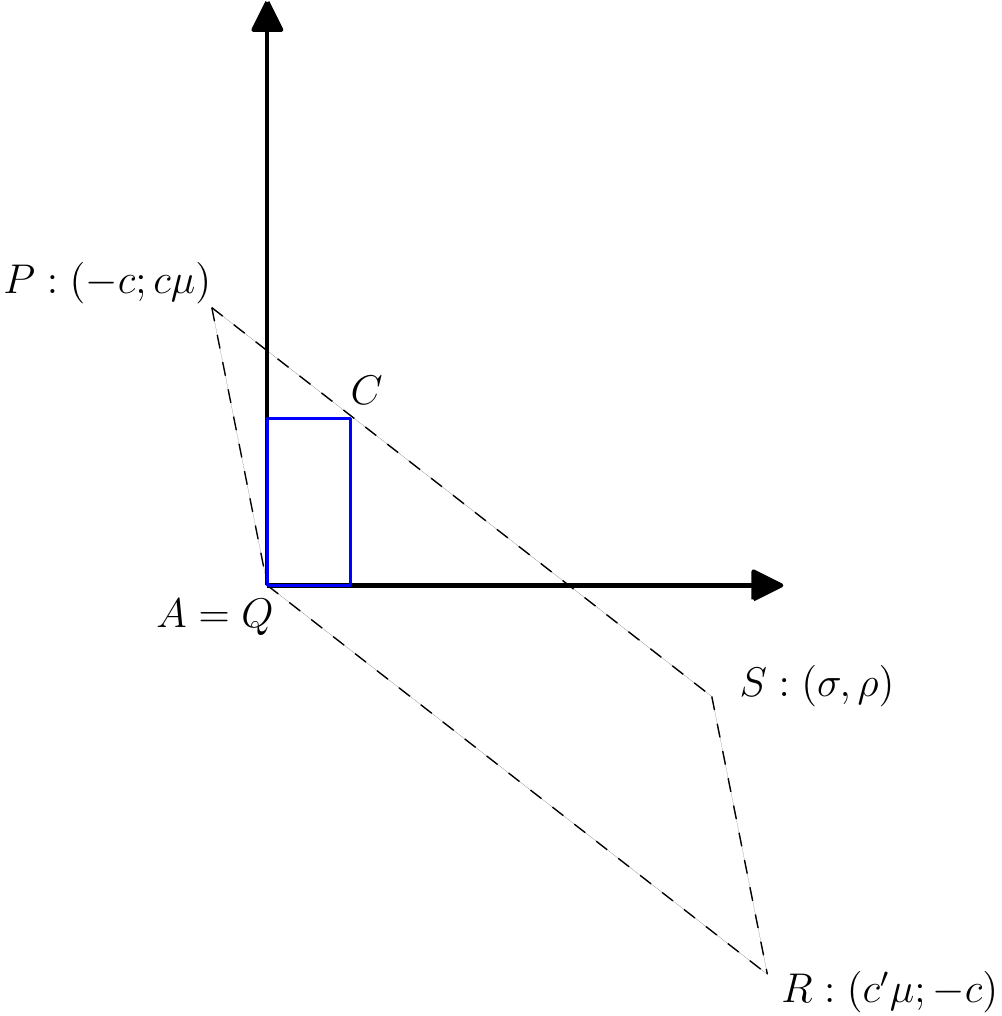}
\end{center}
Then, if $A=Q=(0,0)$, necessarily $C\in [PS]$, \emph{i.e.} $C:=(-c+xc'\mu,c\mu-xc')$ for some $x\in[0,1]$. In this case, ${\sf Area}(\mathcal R)=(-c+xc'\mu)\times (c\mu-xc')$ is always maximal for $x=\frac{c}{2c'}(\mu+\frac{1}{\mu})<1$ and we still get $M=\frac{c^2(\mu^2-1)^2}{4\mu}=\frac{(\rho\mu+\sigma)^2}{4\mu}$.

Similar arguments  give the case $\sigma<0$ but $\rho>0$. Note that we can not have simultaneously $\sigma<0$ and  $\rho< 0$ since  $\mu>1$.
\end{proof}

We now can proceed to the proof of the limiting shape. We set
\begin{equation}\label{eq:Def_cc'mu}
c=\frac{\alpha^{1/4}\beta^{-1/4}b}{\sqrt{\beta-\alpha}},\quad c'=\frac{\beta^{1/4} \alpha^{3/4}a}{\sqrt{\beta-\alpha}},\quad \mu=\sqrt{\frac{\beta}{\alpha}}>1
\end{equation}
and 
\begin{equation}\label{eq:Def_rhosigma}
\sigma=-c+c'\mu,\qquad
\rho=-c'+c\mu.
\end{equation}
We easily check that
\begin{equation}
\rho/\sigma <\mu^{-1} \Longleftrightarrow \frac{b}{a}< \frac{2}{\tfrac{1}{\alpha}+\tfrac{1}{\beta}},\qquad 
\rho/\sigma >\mu  \Longleftrightarrow \frac{b}{a}> \frac{\alpha +\beta}{2}.
\label{eq:Encadrement_ab_versus_sigmarho}
\end{equation}

\begin{proof}[Proof of Theorem \ref{MaxiTheoreme}]

Combining Proposition \ref{Prop:OptiLosange} and Proposition \ref{Prop:Coupling2} we obtain that
$$
f^{\alpha,\beta}(a,b)=\sup\{2\sqrt{{\sf Area}(\mathcal{R}}), \text{ such that the rectangle }\mathcal{R}=[u,u']\times [v,v'] \subset \phi([0,a]\times[0,b])\}
$$
where $\phi$ is defined by \eqref{eq:DefiPhi2}. Hence,  $\phi([0,a]\times[0,b])$  is a parallelogram with vertices 
$$
P:c(-1,\mu),\quad Q:(0,0),\quad R:c'(\mu,-1)),\quad S:(\sigma,\rho),
$$ 
where $c,c',\mu,\sigma,\rho$ are defined by \eqref{eq:Def_cc'mu},\eqref{eq:Def_rhosigma}.
We first collect the results of Lemma \ref{Lem:geometrique_rho} in the case $\rho,\sigma >0$.
If we have $\frac{\rho}{\sigma} <\mu^{-1}$ then Lemma \ref{Lem:geometrique_rho} yields
$$
f^{\alpha,\beta}(a,b) = 2\sqrt{\frac{(\rho\mu+\sigma)^2}{4\mu}} = b\sqrt{\tfrac{1}{\alpha}-\tfrac{1}{\beta}}.
$$
If $\frac{\rho}{\sigma} >\mu $ then again by Lemma \ref{Lem:geometrique_rho}
$$
f^{\alpha,\beta}(a,b) = 2\sqrt{\frac{(\sigma\mu+\rho)^2}{4\mu}} = a\sqrt{\beta-\alpha}.
$$
Finally, in the central case (which corresponds to $\mu^{-1} \leq \rho/\sigma \leq\mu$) we obtain that
$$
f^{\alpha,\beta}(a,b) = 2\sqrt{\sigma\rho} = 2\sqrt{\frac{(-b+\beta a)(-\alpha a+b)}{\beta-\alpha}}.
$$
For the case where $\sigma$ or $\rho$ is negative we observe that
$$
\sigma <0,\rho >0\Leftrightarrow \frac{b}{a} >\beta \left(\geq \frac{\alpha+\beta}{2}\right),\qquad \sigma >0,\rho <0\Leftrightarrow \frac{b}{a} <\alpha \left(\leq \frac{2}{1/\alpha+1/\beta}\right)
$$ 
so the formula claimed in Theorem \ref{MaxiTheoreme} is also valid.
\end{proof}

\subsection{Localization of paths}

This Section is devoted to the proof of Theorem \ref{th:Localization}.
Thanks to the coupling of Proposition \ref{Prop:Coupling2}, it suffices to prove that in the classical Hammersley problem in a parallelogram $t\times \mathcal{P}$, any optimal path concentrates along the diagonal of one of the largest rectangles included in $t\times \mathcal{P}$ (asymptotically with high probability). 

The proof relies on a simple convexity argument for the limiting shape. In the central case the proof is almost identical to that of (\cite{Sepp}, Theorem 2 eq.(1.7)) so we only write the details in the non-central case. (This case is more involved as we know from Lemma \ref{Lem:geometrique_rho} that there are several maximizing rectangles).

\begin{proof}[Proof of Theorem \ref{th:Localization}.]
We fix $\delta$ throughout the proof. As in Lemma \ref{Lem:geometrique_rho},
let $\mu>1$, $c,c'>0$ and denote $P,Q,R,S$ the points of the plane of respective coordinates $(-c,c\mu)$, $(0,0)$,  $(c'\mu,-c')$, $(\sigma,\rho)$ with $\sigma:=-c+c'\mu$ and $\rho:=-c'+c\mu$ so that
the quadrilateral $\mathcal{P}_1:PQRS$ is a parallelogram. Let $\mathcal{P}_t$ be the homothetic transformation of $\mathcal{P}_1$ by a factor $t$, \emph{i.e.}  $\mathcal{P}_{t}$ is the parallelogram with vertices $(-ct,c\mu t)$, $(0,0)$,  $(c'\mu t,-c' t)$, $(\sigma t,\rho t)$.  By symmetry, we can assume without loss of generality that $\rho<\mu^{-1}\sigma$.
Recall from Lemma \ref{Lem:geometrique_rho} that in that case there are several maximizing rectangles in $\mathcal{P}_1$. For all such maximizing rectangles the diagonal has slope $1/\mu$.

Denote $\mathcal{M}_t$ the set of maximizing paths for the classical Hammersley problem in $\mathcal{P}_t$. 
For $\eps>0$ which will be chosen later, we decompose $\mathcal{P}_1$ into disjoint squares of size $\eps\times \eps$:
$$
\mathcal{P}_1\subset\bigcup_{(i,j)\in \mathcal{I}} \bigg\{[\eps i,\eps(i+1))\times [\eps j,\eps(j+1))\bigg\},
$$
for some finite set $\mathcal{I}$. For a real $c$, let $\bar{D}_{c,\delta}$ be the diagonal strip $\bar{D}_{c,\delta} =\set{|\mu y-x-c|\leq \delta/2}$. Thus, with the notations of Theorem \ref{th:Localization}, $\bar{D}_{c,\delta}=(0,c/\mu)+D_{\delta/(2\mu),1/\mu}$.  Let us remark that $\mathcal{P}_1\setminus \bar{D}_{c,\delta}$ is composed of (at most) two connected components, 
\begin{eqnarray*}
\mathcal{Q}^1_{c,\delta}&:=&\{(x,y)\in \mathcal{P}_1, \mu y -x-c>\delta/2\},\\
\mathcal{Q}^2_{c,\delta}&:=&\{(x,y)\in \mathcal{P}_1, \mu y -x-c<-\delta/2\}.
\end{eqnarray*}
 Let $\mathcal{I}^\ell_{c,\delta}$ for $\ell\in\{1,2\}$ be the smallest subset of $\mathcal{I}$ such that
$$
\mathcal{Q}^\ell_{c,\delta}\subset\bigcup_{(i,j)\in \mathcal{I}^\ell_{c,\delta}} \bigg\{[\eps i,\eps(i+1))\times [\eps j,\eps(j+1))\bigg\}
$$
(notations are summarized in Fig. \ref{Fig:ProofLocalization}).
Note that by homogeneity 
$$
\mathcal{P}_t\setminus (t\times \bar{D}_{c,\delta})\subset\left(t\times \mathcal{Q}^1_{c,\delta}\right)\cup \left(t\times \mathcal{Q}^2_{c,\delta}\right).
$$

\begin{figure}
\begin{center}
\includegraphics[width=12cm]{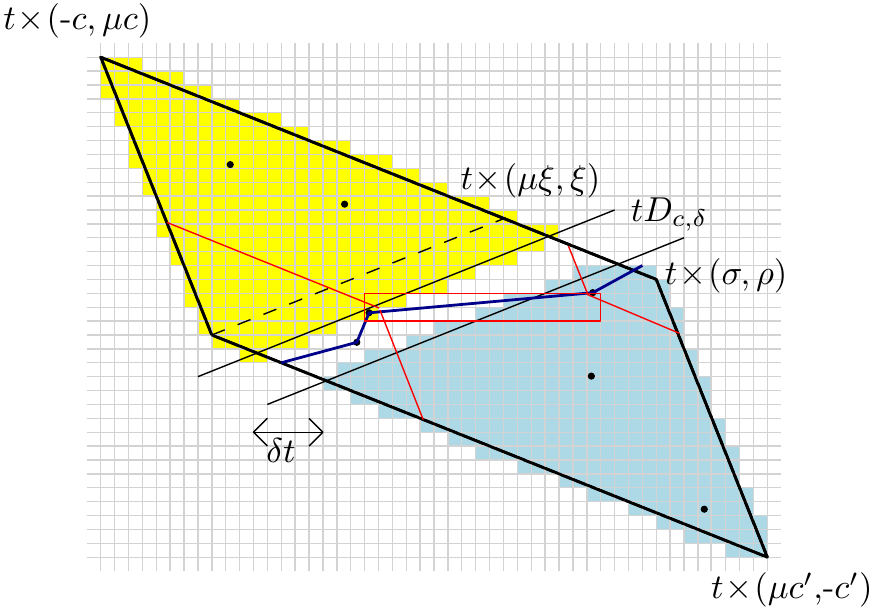}
\caption{Discretization of $\mathcal{P}_t$ into squares of size $t\eps$. In yellow, the small squares indexed by the set $\mathcal{I}^1_{c,\delta}$ and in blue  the one indexed by the set $\mathcal{I}^2_{c,\delta}$. The three parallelograms $\mathcal{G}_{i,j}$, $\tilde{\mathcal{G}}_{k,l}$ and $\mathcal{R}_{i,j,k,l}$ are drawn in red. }
\label{Fig:ProofLocalization}
\end{center}
\end{figure}

If there exists some path $\pi\in \mathcal{M}_t$   not contained in any $t\bar{D}_{c,\delta}$ for any $c\in \R$, this implies in particular that there exists some $c$ such that the path $\pi$ intersects $\mathcal{Q}^1_{c,\delta}$ and $\mathcal{Q}^2_{c,\delta}$. Besides, because of discretization, there exists an integer $n_\eps$ and a finite number of real numbers $c_1,\ldots,c_{n_\eps}$ such that, for any $c\in \R$, $\mathcal{I}^\ell_{c,\delta}=\mathcal{I}^\ell_{c_a,\delta}$, for some $a\in \{1,\ldots,n_\eps\}$.
Thus we have
\begin{multline}\label{eq:Event2}
\bigg\{\exists \pi\in \mathcal{M}_t\text{ not contained in }t\bar{D}_{c,\delta}\text{ for some }c\bigg\}\\
\subset \bigcup_{a=1}^{n_\eps} \bigcup_{(i,j)\in \mathcal{I}^1_{c_a,\delta}}  \bigcup_{(k,l)\in \mathcal{I}^2_{c_a,\delta}} \bigg\{\exists \pi \mbox{ intersecting }\  [t\eps i,t\eps(i+1))\times [t\eps j,t\eps(j+1)) \mbox{ and } [t\eps k,t\eps(k+1))\times [t\eps l,t\eps(l+1)) \bigg\}.
\end{multline}
Let 
 $a\in \{1,\ldots,n\}$, $(i,j)\in \mathcal{I}^1_{c_a,\delta}$ and $(k,l)\in \mathcal{I}^2_{c_a,\delta}$. We need to prove that  the probability that there exists a maximizing path  which intersects  the squares $[t\eps i,t\eps(i+1))\times [t\eps j,t\eps(j+1))$ and $[t\eps k,t\eps(k+1))\times [t\eps l,t\eps(l+1))$ tends to 0 as $t$ tends to infinity. By symmetry, we can assume that $i\le k$ and $j\le l$ as in  Figure \ref{Fig:ProofLocalization}.

Let $\mathcal{G}_{i,j}$ be the parallelogram with opposite vertices $(0,0)$ and $\varepsilon (i+1,j+1)$ and with sides parallel to the sides of $\mathcal{P}$. Let $\tilde{\mathcal{G}}_{k,l}$ be the parallelogram with opposite vertices  $\varepsilon (k,l)$  and $(\sigma, \rho)$ and with sides parallel to the sides of $\mathcal{P}$. Denote $\mathcal{R}_{i,j,k,l}$ the rectangle  with sides parallel to the axes and with  opposite vertices  $\varepsilon (i,j)$  and $\varepsilon (k+1,l+1)$.  Thus we get 
\begin{multline*}
 \bigg\{\exists \pi \mbox{ intersecting }  [t\eps i,t\eps(i+1))\times [t\eps j,t\eps(j+1)) \mbox{ and } [t\eps k,t\eps(k+1))\times [t\eps l,t\eps(l+1)) \bigg\}\\
\subset \bigg\{L(t\mathcal{G}_{i,j})+L(t\tilde{\mathcal{G}}_{k,l})+ L(t\mathcal{R}_{i,j,k,l})\ge L(\mathcal{P}_t) \bigg\}.
\end{multline*}
We can compute the almost sure limit of  $\left(L(t\mathcal{G}_{i,j})+L(t\tilde{\mathcal{G}}_{k,l})+L(t\mathcal{R}_{i,j,k,l})\right)/t$ using  Proposition \ref{Prop:OptiLosange} and Lemma \ref{Lem:geometrique_rho}. 
We let the reader check that,  if $\eps$ is chosen small enough, for any $(i,j)\in \mathcal{I}^1_{c,\delta}$ and $(k,l)\in \mathcal{I}^2_{c,\delta}$,  this limit is strictly smaller than the limit of $L(\mathcal{P}_t)/t$. This computation is very similar to the proof of Lemma \ref{Lem:geometrique_rho}.
It is a consequence of the fact that the function
$$
\begin{array}{r c l}
(\mathcal{P}_t)^2 & \to & \mathbb{R}^2\\
(a,b), (\alpha,\beta) &\mapsto & \tilde{f}(a,b)+ \tilde{f}(\alpha-a,\beta-b) + \tilde{f}(\sigma-\alpha,\rho-\beta)
\end{array}
$$
(where $\tilde{f}(a,b)=2\sqrt{ab}\mathbf{1}_{a,b>0}$) is strictly maximized when $(a,b)$ and $(\alpha,\beta)$ belong to a straight line with slope $\mu^{-1}$.
Thus, we deduce that
$$\lim_{t \to \infty} \P\bigg\{L(t\mathcal{G}_{i,j})+L(t\tilde{\mathcal{G}}_{k,l})+ L(t\mathcal{R}_{i,j,k,l})\ge L(\mathcal{P}_t) \bigg\}=0.$$
Therefore the event on the LHS of \eqref{eq:Event2} is included in a finite union of events (and the number of those events does not depend on $t$) whose probabilities tend to zero.

This proves that in the classical Hammersley problem in a parallelogram optimal paths concentrate along one of the maximizing diagonals. The proof of Theorem \ref{th:Localization} follows from the coupling of Proposition \ref{Prop:Coupling2}.
\end{proof}

\section{Fluctuations}
\subsection{The standard case : Proof of Theorem \ref{Th:FluctuLipschitz} (i) }

We first recall some known facts regarding the fluctuations for the classical Hammersley problem.


\begin{prop}\label{Prop:Flu_Ln_continu}
For every $m>0$, there exists $C_m>0$ such that for every rectangle $\mathcal{R}$  with sides parallel to the axes and every $\eta >0$
\begin{equation}\label{eq:MajoLtbis}
\P\left( L(\mathcal{R})\ge 2\sqrt{{\sf Area}(\mathcal{R})}+  {\sf Area}(\mathcal{R})^{1/6+\eta }\right)\leq \frac{C_m}{{\sf Area}(\mathcal{R})^{\eta m}}.
\end{equation}
\end{prop}
\begin{proof}
The claimed inequality is obtained by applying Markov's inequality to (\cite[Th.1.2]{BDK})\footnote{Theorem 1.2 in \cite{BDK} is actually stated for longest increasing sequences in permutations but also holds for the Poissonized version which coincides with our $L(.)$, as explained in Section 8 of \cite{BDK}. See also \cite[Ch.2]{Romik} for details about de-Poissonization.}.
\end{proof}
With the help of Proposition \ref{Prop:BDJ}, we will have proved Theorem \ref{Th:FluctuLipschitz} as soon as we establish that, inside a large parallelogram $\mathcal{P}$ with a unique optimal rectangle $\mathcal{R}$, the fluctuations of $L(\mathcal{P})$ are equal to the fluctuations of $L(\mathcal{R})$, \emph{i.e.}
 non-optimal rectangles do not modify them. To this end, we show the following proposition.

\begin{prop}\label{Prop:flutupara} Let $\mu>1$ and  $\sigma, \rho>0$.
Let us denote $\mathcal{P}_1=PQRS$ the parallelogram with $Q=(0,0)$, 
$S=(\sigma,\rho)$ and such that $P=c(-1,\mu)$
and $R=c'(\mu,-1)$ for some $c,c'>0$. For $t\ge 0$, let $\mathcal{P}_t$ be the homothetic transformation of $\mathcal{P}_1$ by a factor $t$.
Assume $\mu^{-1}< \rho/\sigma < \mu$. Then
$$
\lim_{t \to \infty}\P(L({\mathcal{P}_t})\ge L(\sigma t,\rho t) +t^{\frac{1}{3}-\frac{1}{32}})=0.
$$
\emph{(Observe that we always have the inequality $L({\mathcal{P}_t})\ge L(\sigma t,\rho t)$.)}
 \end{prop}

Proposition \ref{Prop:flutupara}  says that the difference between  $L({\mathcal{P}_t})$ and $ L(\sigma t,\rho t)$ is a little-o of the fluctuations.
Let us prove formally that it implies Theorem \ref{Th:FluctuLipschitz}.

\begin{proof}[Proof of Proposition \ref{Prop:flutupara} $\Rightarrow$ Theorem \ref{Th:FluctuLipschitz}]
 Using  Proposition \ref{Prop:Coupling2}, we know that $L^{\alpha,\beta}(at,bt)$ has the law of $L({\mathcal{P}_t})$ where $\mathcal{P}_t$ is the homothetic transformation of $PQRS$ by a factor $t$, where
$$
P:c(-1,\mu),\quad Q:(0,0),\quad R:c'(\mu,-1)),\quad S:(\sigma,\rho),
$$
and $c,c',\mu,\sigma,\rho$ are defined by \eqref{eq:Def_cc'mu},\eqref{eq:Def_rhosigma}.

Recall from \eqref{eq:Encadrement_ab_versus_sigmarho} that the case
$
\displaystyle{\frac{2}{1/\alpha+1/\beta} < \frac{b}{a}<\frac{\beta+\alpha}{2}}
$
corresponds to $0< \mu^{-1}< \rho/\sigma< \mu$.
Thus, $\mathcal{P}_t$ is an homothetic transformation of a parallelogram satisfying the condition of Proposition \ref{Prop:flutupara}. Hence, we get that
\begin{equation*}
\frac{L^{\alpha,\beta}(at,bt) -L(\sigma t,\rho t)}{t^{1/3}}
\end{equation*}
converges in probability to $0$.
 We conclude using Proposition \ref{Prop:BDJ} and noticing that $2\sqrt{\sigma \rho}=f^{\alpha,\beta}(a,b)$.
\end{proof}

\begin{proof}[Proof of Proposition \ref{Prop:flutupara}]

To prove Proposition \ref{Prop:flutupara} we need to bound, for any rectangle of the form
$$
\mathcal{R}_{uv}:=[-u,\sigma t+v]\times[\mu^{\mbox{\tiny{sign(u)}}}u,\rho t - \mu^{\mbox{\tiny{sign(v)}}}v]
$$
inside $\mathcal{P}_t$ (and with two vertices on the boundary of $\mathcal{P}_t$), the probability that $L(\mathcal{R}_{uv})-L(\sigma t,\rho t)$ is larger than $t^{1/3-1/32}$. In fact, we will establish the following equation
\begin{equation}\label{Eq:supsurrectangle}
\lim_{t \to \infty}\P(\sup_{(u,v)\in\mathcal{U}} L(\mathcal{R}_{uv})-L(\sigma t,\rho t)\ge t^{1/3-1/32})=0
\end{equation}
where $\mathcal{U}$ is the set of value $(u,v)$ for which $\mathcal{R}_{uv}$ is inside $\mathcal{P}_t$.
To this end, we set 
\begin{equation*}
\kappa=\frac{1}{3}+\frac{1}{32}
\end{equation*}
and we  treat separately two cases.
\begin{itemize}
\item[(A)] Either $|u|>t^{\kappa}$ or $|v|>t^\kappa$. Then the rectangle $\mathcal{R}_{uv}$ is sufficiently smaller than the optimal rectangle $[0,\sigma t]\times[0,\rho t]$ such that the law of large number for $L(\mathcal{R}_{uv})$ will insure that it is very unlikely that $L(\mathcal{R}_{uv})\ge 2\sqrt{\sigma\rho} t -t^{1/3+\eta}$ (for some $\eta$ to be determined).
\item[(B)] Or $|u|<t^{\kappa}$ and $|v|<t^\kappa$.  Then $\mathcal{R}_{uv}$ and $[0,\sigma t]\times [0,\rho t]$ are so close to each other that necessarily $L(\mathcal{R}_{uv})-L(\sigma t,\rho t)$ can not be too large.
\end{itemize}

\subsubsection*{The case (A): $|u|>t^{\kappa}$ or $|v|>t^\kappa$}
 By symmetry, we only consider here the case $u>t^{\kappa}$ and $v\ge 0$ but the reader can easily adapt the proof to the other cases.
 As in the proof of Proposition \ref{Prop:OptiLosange}, we discretize the problem. More precisely, let $u_i=v_i=i/t^2$ and denote 
 $\mathcal{I}$ the set of indices $(i,j)$ such that $\mathcal{R}_{u_iv_j}$ satisfies $u_i>t^{\kappa}$ and $v_j\ge 0$ and is inside $\mathcal{P}_t$ namely
 $$\mathcal{I}=\{(i,j)\in \Z_+^2, u_i>t^\kappa, (u_i, v_j)\in\mathcal{U}\}.$$

 As in the proof of Proposition \ref{Prop:OptiLosange}, we have, with probability tending to 1 as $t$ tend to infinity
 \begin{equation}\label{Eq:discegaltou}
 \sup \{L(\mathcal{R}_{uv}),(u,v)\in\mathcal{U},u>t^{\kappa},v\ge 0\}=\sup \{L(\mathcal{R}_{u_iv_j}),(i,j)\in\mathcal{I}\}
 \end{equation}
and we have   $|\mathcal{I}|\le Nt^6$ for some constant $N:=N(\sigma,\rho,\mu)$. Besides,
 we have (some insight for the below computation will be given in  Remark \ref{rem:KPZ})
 \begin{align*}
{\sf Area}(\mathcal{R}_{u_iv_j})
&={\sf Area}\left(  [-u_i,\sigma t+v_j]\times[\mu u_i,\rho t - \mu v_j] \right)  \\
&=\sigma\rho t^2 -t(u_i+v_j)(\sigma \mu-\rho)-\mu(u_i+v_j)^2\le \sigma \rho t^2 -t^{\kappa+1}(\sigma\mu-\rho),
\end{align*}
since $u_i >t^\kappa$. Hence 
$$\sqrt{\sigma\rho} t-\sqrt{{\sf Area}(\mathcal{R}_{u_iv_j})}= \frac{\sigma\rho t^2- {\sf Area}(\mathcal{R}_{u_iv_j})}{\sqrt{\sigma\rho} t+\sqrt{{\sf Area}(\mathcal{R}_{u_iv_j})}} \geq  \frac{\sigma\rho t^2- {\sf Area}(\mathcal{R}_{u_iv_j})}{2\sqrt{\sigma\rho} t} \geq \frac{\sigma\mu-\rho}{2\sqrt{\sigma\rho}}t^{\kappa}.
$$
Thus for $\eta >0$
$$
 \P(L(\mathcal{R}_{u_iv_j})\ge 2\sqrt{\sigma\rho} t -t^{1/3+\eta})\le \P(L(\mathcal{R}_{u_iv_j})-2\sqrt{{\sf Area}(\mathcal{R}_{u_iv_j})}\ge \frac{\sigma\mu-\rho}{\sqrt{\sigma\rho}}t^{\kappa}-t^{1/3+\eta}).
$$
Choose $\eta=1/65$ such that
  $1/3+2\eta<\kappa$ and so that $t^\kappa$ is asymptotically larger than ${\sf Area}(\mathcal{R}_{u_iv_j})^{1/6+\eta}$. Since  by assumption $\sigma \mu-\rho>0$ then Proposition \ref{Prop:Flu_Ln_continu} with $m=7/2\eta$ yields the existence of a constant $\tilde{C}_m:=\tilde{C}_m(\sigma,\rho,\mu)$ such that
 $$
\forall (i,j)\in \mathcal{I}, \quad\P(L(\mathcal{R}_{u_iv_j})\geq 2\sqrt{{\sf Area}(\mathcal{R}_{u_iv_j})} + \frac{\sigma\mu-\rho}{\sqrt{\sigma\rho}}t^{\kappa})
\le \frac{C_m}{\left({\sf Area}(\mathcal{R}_{u_iv_j})\right)^{m\eta }}
\le \frac{\tilde{C}_m}{t^{7}}
$$
 and finally
  $$
\P(\sup_{(i,j)\in \mathcal{I}}L(\mathcal{R}_{u_iv_j})\ge 2\sqrt{\sigma\rho} t -t^{1/3+\eta})\le |\mathcal{I}|\frac{\tilde{C}_m}{t^{7}} \leq  \frac{N \tilde{C}_m}{t}.
$$
Case $u\geq t^\kappa,v\leq 0$ is treated in the same way and the cases where $u<-t^\kappa$ hold by symmetry. Besides, Proposition \ref{Prop:Flu_Ln_continu} implies that
$$\lim_{t \to \infty}\P(L(\sigma t,\rho t)\le 2\sqrt{\sigma\rho} t -t^{1/3+\eta})=0.$$
Combining this with \eqref{Eq:discegaltou}  we get 
 \begin{equation}\label{eq_revision1}
\lim_{t \to \infty}\P(\sup\{L(\mathcal{R}_{uv}),(u,v)\in\mathcal{U},|u|>t^{\kappa} \mbox{ or } |v|>t^\kappa\}-L(\sigma t,\rho t)\ge 0)=0
\end{equation}
 which finishes the case (A) of our proof.

\begin{remark}\label{rem:KPZ}{\bf (Discrepancy with the KPZ scaling)}   In order to establish the relation $\chi=2\xi -1$ between the scaling  exponents $\chi$ and $\xi$ it is natural (see \cite{JohanssonTransversal} to make this rigorous) to estimate the quantity
$$
L(\sigma t, \rho t) - L(\sigma (t+t^\xi),\rho (t-t^\xi))\approx 2\sqrt{\sigma \rho t^2}-2\sqrt{\sigma \rho(t+t^\xi)(t-t^\xi)}\approx \sqrt{\sigma \rho}t^{2\xi -1}, 
$$ 
which should be of the same order as the fluctuations $t^\chi$. In the proof of case (i) we need to consider a perturbation along another  (unusual) direction and therefore the leading term in the above approximation is of different order. Indeed (recall $\mu >\rho/\sigma$)
$$
L(\sigma t, \rho t) - L(\sigma (t+t^\xi),\rho (t-  t^\xi \mu\sigma/\rho))
\approx 2\sqrt{\sigma \rho t^2}-2\sqrt{\sigma \rho(t+t^\xi)(t-\tfrac{\mu\sigma}{\rho}t^\xi)}
\approx  \sqrt{ \sigma\rho (1-\tfrac{\mu\sigma}{\rho}  )} \ t^{\xi}.
$$ 
\end{remark}
 
\subsubsection*{The case (B): $|u|<t^{\kappa}$ and $|v|<t^\kappa$}
This case is more delicate than the previous one. 
We discretize again the boundary of $\mathcal{P}_t$ but with a different scaling than before. All along case (B) we put
$$
\delta=\frac{1}{3}-\frac{1}{16} \qquad \mbox{ and }\qquad  u_i=v_i=it^\delta \quad \mbox{ for }\quad  |i|,|j|  \le t^{\kappa-\delta}.
$$  

The main ideas of the proof of case (B) are inspired by the proof of Theorem 2 in \cite{CatorPimentel} and require some  important tools regarding the Hammersley process. In particular, we will use 
\begin{itemize}
\item the graphical representation of \emph{Hammersley lines} (implicitly introduced by Hammersley
 \cite{HammersleyHistorique}) ;
\item the \emph{sinks/sources} approach introduced in \cite{CatorGroeneboom}.
\end{itemize}

\paragraph{Sinks and sources.}
\begin{figure}
\begin{center}
\includegraphics[width=9cm]{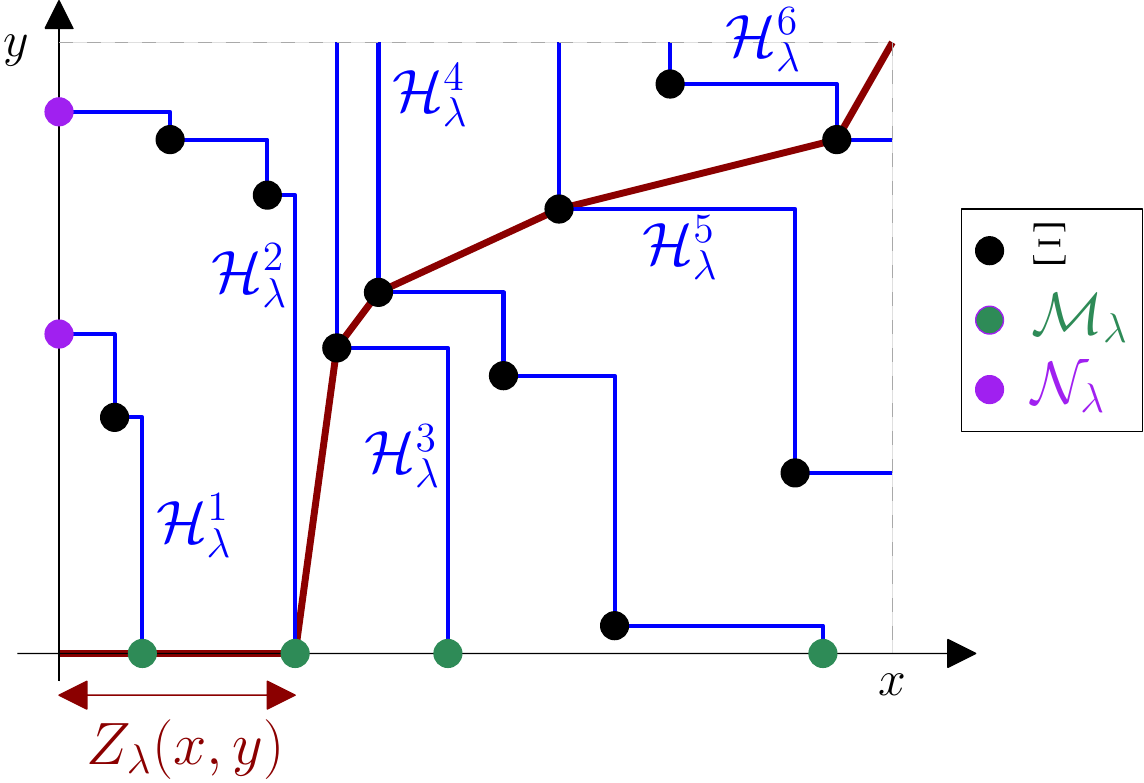}
\caption{Summary of notations: sources, sinks, Hammersley lines. On this sample we have $L_{\lambda} (x,y)=6$ and an optimal path (drawn in dark red) goes through two sources.}
\label{Fig:NotationsSourcesPuits}
\end{center}
\end{figure}

We introduce the Hammersley process with \emph{sinks} and \emph{sources}. 
For  $\lambda>0$ we consider the point process $\Xi^{\lambda}$ on the plane defined by
$$
\Xi^{\lambda}:= \Xi \cup \mathcal{M}_\lambda \cup \mathcal{N}_{\lambda}
$$
where $\mathcal{M}_\lambda$ is a one dimensional Poisson point process of intensity $\lambda$ on the $x$-axis, $\mathcal{N}_{\lambda}$ is a  one dimensional Poisson point process of intensity $\lambda^{-1}$ on the $y$-axis (and $\Xi,\mathcal{M}_\lambda,\mathcal{N}_{\lambda}$ are independent). A point on $\mathcal{M}_\lambda$ (resp. on $\mathcal{N}_{\lambda}$) is called a source (resp. a sink).

We will modify $L$ by considering additional points of $\mathcal{M}_\lambda \cup \mathcal{N}_{\lambda}$:
\begin{multline*}
L_{\lambda} (x,y)=
\max\bigg\{L; \hbox{ there are }(u_1,v_1),\dots,(u_L,v_L)\in \Xi \cup \mathcal{M}_\lambda \cup \mathcal{N}_{\lambda} \text{ such that }\\
(0,0) \preccurlyeq (u_1,v_1) \preccurlyeq (u_2,v_2)  \preccurlyeq \dots \preccurlyeq (u_L,v_L) \preccurlyeq (x,y)
\bigg\}.
\end{multline*}

\paragraph{Hammersley lines.}
 For $1\leq \ell \leq L_{\lambda}(x,y)$ the $\ell$-th \emph{Hammersley line} $\mathcal{H}^\ell_{\lambda}$  is defined as the lowest level set of level $\ell$ for the function $(x,y)\mapsto L_{\lambda}(x,y)$ (see Fig. \ref{Fig:NotationsSourcesPuits}, Hammersley lines are drawn in blue). It is easy to see that by construction 
every Hammersley line is a simple curve composed of an alternation of South/East straight lines.

It was shown and exploited by Cator and Groeneboom \cite{Groeneboom,CatorGroeneboom} that for the choice of intensities $\lambda$ (sources) and $\lambda^{-1}$ (sinks) the process of Hammersley lines is stationary either seen from bottom to top or from left to right. 
As a consequence we obtain the following non-asymptotic estimate.
\begin{lem}[\cite{CatorGroeneboom}, Theorem 3.1]\label{lem:SourcesPuitsStationnaires}
\ 
For every $0<x_0<x_1$ and $0<y_0<y_1$ we have: 
\begin{itemize}
\item {\bf Stationarity of sources.}
$$
\mathcal{S}_1:=L_{\lambda}(x_1,y_0)- L_{\lambda}(x_0,y_0) \stackrel{\text{(d)}}{=} \text{Poisson r.v. with mean }\lambda(x_1-x_0).
$$
\item {\bf Stationarity of sinks.} For every $x>0$ and every $y_0<y_1$,
$$
\mathcal{S}_2:=L_{\lambda}(x_0,y_1)- L_{\lambda}(x_0,y_0) \stackrel{\text{(d)}}{=} \text{Poisson r.v. with mean }\frac{1}{\lambda}(y_1-y_0).
$$
\item Furthermore, $\mathcal{S}_1$ and $\mathcal{S}_2$ are independent. 
\end{itemize}

\end{lem}

We now introduce the random variable $Z_\lambda(x,y)$ which represents the maximal length that an optimal path spends on the $x$-axis:
$$
Z_\lambda(x,y) = \sup \set{\xi\geq 0 ; L_{\lambda}(x,y)= \mathrm{card}(\mathcal{M}_\lambda \cap [0,\xi]) + L([\xi,x]\times [0,y])  },
$$
where $\sup \varnothing = 0$.

\begin{lem} \label{Lem:decol}
There exists some constant $C_1:=C_1(\sigma,\rho)$ such that, for any $t>0$, 
$$
\mathbb{P}(Z_{\lambda^\star_+}(\sigma t,\rho t)= 0)\leq \frac{C_1}{t^{1/4}}
$$
with $\lambda^\star_+:=\sqrt{\rho/\sigma}(1+t^{-1/4})$.
\end{lem}

\begin{proof}
Lemma 6 of \cite{CatorPimentel} states that there exists some constant $C>0$ such that, for any $t\ge 0$,
$$\P(Z_{1+t^{-1/4}}(t, t) =0)\leq \frac{C}{t^{1/4}}.$$
(In \cite[Lemma 6]{CatorPimentel} we are not interested in $Z$ and $\gamma$ but only in $Z',\gamma'$.)
We just need here to rewrite their result in order to obtain the analogue for a rectangle. Using the invariance of a Poisson process  under the transformation  $(x,y)\mapsto (x\sqrt{\rho/\sigma},y\sqrt{\sigma/\rho})$, we get the equality in distribution:
$$
Z_{\lambda^\star_+}(\sigma t,\rho t)\overset{\mbox{\tiny law}}{=}Z_{\lambda^\star_+\sqrt{\sigma/\rho}}(\sqrt{\sigma\rho} t,\sqrt{ \sigma\rho} t).
$$
Since
$\lambda_+^\star=\sqrt{\rho/\sigma}(1+t^{-1/4})$
we get 
$$
\mathbb{P}(Z_{\lambda^\star_+}(\sigma t,\rho t)= 0)\leq \P(Z_{1+t^{-1/4}}(\sqrt{\sigma \rho}t, \sqrt{\sigma \rho}t)= 0)\le \frac{C}{(\sigma\rho)^{1/8}t^{1/4}}=: \frac{C_1}{t^{1/4}}.$$

\end{proof}

We now state a deterministic lemma, this kind of inequality is known as "Crossing Lemmas" in the literature.
\begin{lem}\label{Lem:domisource} For any $y\ge x$, $s\le t$ and $\lambda>0$, if the event $Z_\lambda(x,t)>0$ occurs then we have
$$
L(y,s)-L(x,t)\le L_{\lambda}(y,s)- L_{\lambda}(x,t).
$$
\end{lem}

\begin{proof}
For $s=t$ this is Lemma 2 in \cite{CatorPimentel}. The diagram below explains how it can be extended for every $s\le t$: the dashed line represents an optimal path for $L_\lambda(x,t)$ while the plain line  represents an optimal path for $L(y,s)$. On the event $Z_\lambda(x,t)>0$ the dashed line crosses the plain line and these paths can be decomposed respectively as  \ding{193}$+$\ding{195} and  \ding{192}$+$\ding{194} as follows:
\begin{center}
\includegraphics[width=6cm]{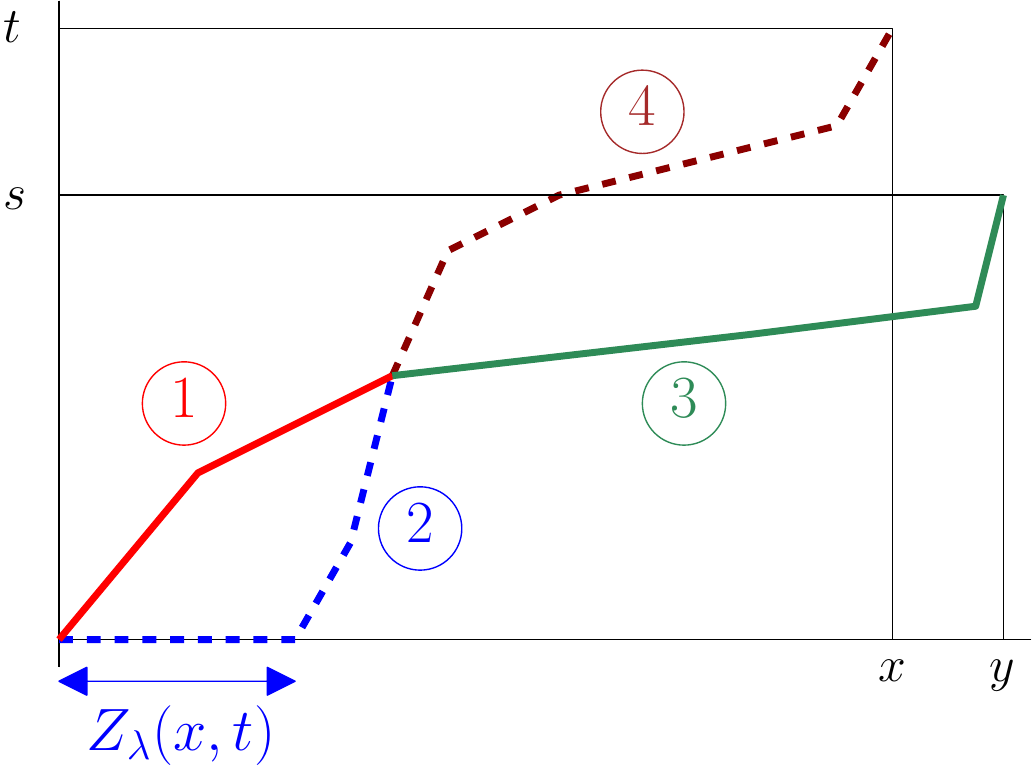}
\end{center}
We have that
$$
L(y,s)=\text{\ding{192}}+\text{\ding{194}},\qquad L(x,t)\geq \text{\ding{192}}+\text{\ding{195}}
$$
from which we deduce $L(y,s)-L(x,t)\leq \text{\ding{194}}-\text{\ding{195}}$. On the other hand
$$
L_{\lambda}(y,s)\geq \text{\ding{193}}+\text{\ding{194}},\qquad L_\lambda(x,t)= \text{\ding{193}}+\text{\ding{195}}
$$
from which we deduce  $L_\lambda(y,s)-L_\lambda(x,t)\geq \text{\ding{194}}-\text{\ding{195}}$.
\end{proof}

We will use the following exponential concentration inequalities (see Sect.2.4.1 in \cite{Remco}):
\begin{lem}\label{lem:poisson}
If $X$ is a Poisson random variable with mean $\theta$ and $a>0$ then
$$
\P(X  \geq \theta+a )\le \exp\left( -\frac{a^2}{2(\theta+a)}\right),
\qquad \P(X  \leq \theta-a )\le \exp\left( -\frac{a^2}{2(\theta+a)}\right). 
$$As a consequence if $X,Y$ are Poisson random variables with respective means $\theta<\theta'$:
 \begin{equation}\label{eq:ConcentrationDeuxPoisson1}
\P(X\geq Y )\le 2\exp\left(- \frac{(\theta'-\theta)^2}{12\theta'}\right).
\end{equation}
\end{lem}

We  establish Proposition \ref{Prop:flutupara} by proving the following lemma.
\begin{lem}\label{Lem:poissongene} 
Assume $\mu^{-1}<\rho/\sigma< \mu$. There exists $C_2$ (depending on $\sigma,\rho,\mu$) such that, for any $i,j\in [-t^{\kappa-\delta},t^{\kappa-\delta}]$ and $t\ge 1$,
\begin{equation}\label{Eq:fluctu1}
\mathbb{P}\left( L(\mathcal{R}_{u_iv_j}) \geq L(\sigma t,\rho t) \right)
\leq \frac{C_2}{t^{1/4}}.
\end{equation}
\end{lem}

\begin{proof} We begin by proving the lemma for $i=0$ and $j\ge 0$ \emph{i.e.} showing that there exists  $C_2$ such that, for any $j\in \Z_+$ and $t\ge 1$,
\begin{equation}\label{Eq:fluctu_i=0}
\mathbb{P}\left( L(\sigma t +jt^\delta, \rho t-\mu jt^\delta) \geq L(\sigma t,\rho t) \right)
\leq \frac{C_2}{t^{1/4}}.
\end{equation}
As in Lemma \ref{Lem:decol}, let $\lambda^\star_+=\sqrt{\rho/\sigma}(1+t^{-1/4})$.
Using Lemma \ref{Lem:domisource}, we have
\begin{eqnarray}
\mathbb{P}\left( L(\mathcal{R}_{0v_j}) \geq L(\sigma t,\rho t) \right)&\leq &\mathbb{P}(Z_{\lambda^\star_+}[\sigma t,\rho t]= 0)+\mathbb{P}\left( L_{\lambda^\star_+}(\mathcal{R}_{0v_j}) \geq L_{\lambda^\star_+}(\sigma t,\rho t) \right)\notag\\
&\leq & \frac{C_1}{t^{1/4}}+ \mathbb{P}\left( L_{\lambda^\star_+}(\mathcal{R}_{0v_j}) \geq L_{\lambda^\star_+}(\sigma t,\rho t) \right).\label{eq:ComparaisonPoisson}
\end{eqnarray}
In order to bound the last term in the RHS of \eqref{eq:ComparaisonPoisson} we write
$$
L_{\lambda^\star_+}(\mathcal{R}_{0v_j})-  L_{\lambda^\star_+}(\sigma t,\rho t)=
\left(L_{\lambda^\star_+}(\mathcal{R}_{0v_j})-L_{\lambda^\star_+}(\sigma t, \rho t-\mu jt^\delta)\right)- \left( L_{\lambda^\star_+}(\sigma t,\rho t)-L_{\lambda^\star_+}(\sigma t , \rho t-\mu jt^\delta)\right).
$$
Using the stationarity of the Hammersley process with sources and sink (Lemma \ref{lem:SourcesPuitsStationnaires}), we know that  both terms in the RHS of the above equation has a Poisson distribution. More precisely, 
\begin{equation}\label{eq:DomiPoisson}
\mathbb{P}\left( L_{\lambda^\star_+}(\mathcal{R}_{0v_j}) \geq L_{\lambda^\star_+}(\sigma t,\rho t)\right)
= 
\mathbb{P}\left(\mathrm{Poisson}(\lambda^\star_+ jt^\delta) \geq \mathrm{Poisson}(\frac{\mu}{\lambda^\star_+}jt^\delta)  \right),
\end{equation}
where the two Poisson random variables on the RHS are independent.

For $t$ large enough, $\lambda^\star_+$ tends to $\sqrt{\rho/\sigma}$. Therefore the mean of the first Poisson r.v. in \eqref{eq:DomiPoisson} becomes strictly smaller than the second mean (here we use $\rho/\sigma <\mu$). Hence, \eqref{eq:ConcentrationDeuxPoisson1} implies the existence of a constant $C_2:=C_2(\sigma,\rho,\mu,j)$ satisfying \eqref{Eq:fluctu_i=0}.

This concludes the proof of Lemma \ref{Lem:poissongene} for $i=0,j>0$. The proof of the case $i=0,j<0$ is identical (using $\mu^{-1}<\rho/\sigma$).
To extend the formula to $i\neq 0$, we write
$$L(\mathcal{R}_{u_iv_j})- L(\sigma t,\rho t)=
 \Big(L(\mathcal{R}_{u_iv_j})- L(\mathcal{R}_{0v_j})\Big)+\Big(L(\mathcal{R}_{0v_j})-L(\sigma t,\rho t)\Big)$$
Using the same argument as before in the rectangle $\mathcal{R}_{0v_j}$ instead of $\mathcal{R}_{00}$ we still have that $\lambda^\star_+$ (which depends on $j$) is $\sqrt{\rho/\sigma}+o(1)$ uniformly in $j$  
(here we use that $|jt^\delta|\le t^{\kappa}\ll t$). 
Therefore we can bound the probability of the event $ \{L(\mathcal{R}_{u_iv_j})- L(\mathcal{R}_{0v_j})\ge 0\}$  uniformly for $j\in [-t^{\kappa-\delta},t^{\kappa-\delta}]$. 

\end{proof}

We have now all the ingredients to finish the proof of Proposition \ref{Prop:flutupara} \emph{i.e.} to prove \eqref{Eq:supsurrectangle}.
 Recall that $\mathcal{U}$ denote the set of value $(u,v)$ for which $\mathcal{R}_{uv}$ is inside $\mathcal{P}_t$.
 Using the conclusion of Case A (Eq.\eqref{eq_revision1}) and  the symmetry, it remains to bound  $ \sup\{ L(\mathcal{R}_{u v})-L(\sigma t,\rho t), (u,v)\in\mathcal{U}, 0\leq u<t^\kappa, 0\leq v<t^\kappa  \}$.
Since Card$\{|i|,|j|\le t^{\kappa-\delta}\}\sim 4t^{2\kappa-2\delta}=4t^{3/16}=o(t^{1/4})$,
we get using Lemma \ref{Lem:poissongene} that
\begin{equation}
\lim_{t \to \infty}\mathbb{P}\left( \sup_{|i|,|j|\le t^{\kappa-\delta}} L(\mathcal{R}_{u_iv_j}) \geq L(\sigma t,\rho t) \right)=0.
\end{equation}
For $u\in [u_i,u_{i+1}]$, $v\in [v_j,v_{j+1}]$ consider a maximizing path in $\mathcal{R}_{uv}$. This path leaves the rectangle $\mathcal{B}^i:=[-u_{i+1},-u_{i}]\times[\mu u_{i},\mu u_{i+1}]$ through the top or right border and enters the rectangle  $\mathcal{C}^j:=[\sigma t+v_j,\sigma t+v_{j+1}]\times[\rho t-\mu v_{j+1},\rho t-\mu v_j]$ through the left or the bottom. One of those four cases is illustrated in Fig.\ref{Fig:Ruv}.
\begin{figure}[h]
\begin{center}
\includegraphics[width=13cm]{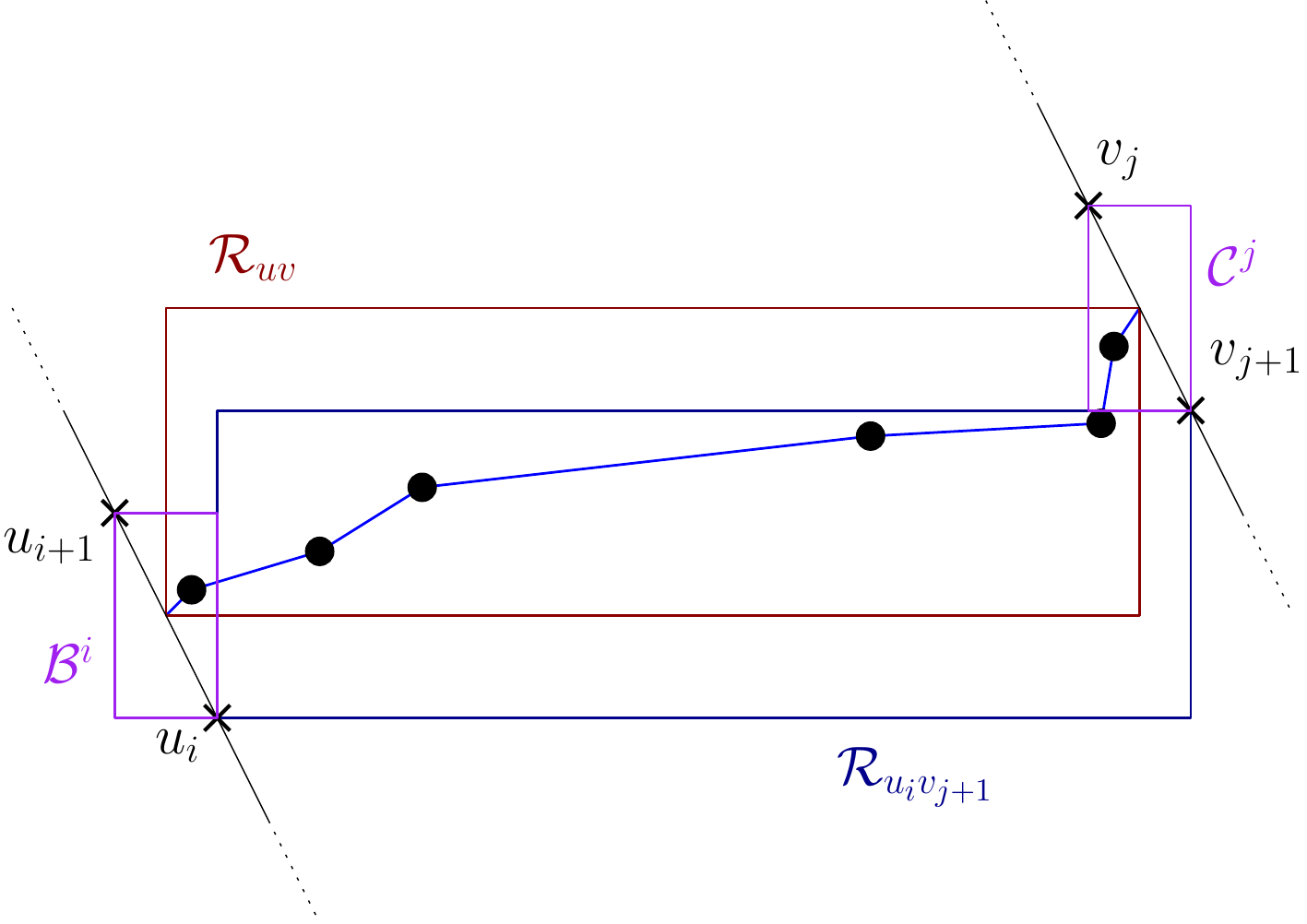}
\end{center}
\caption{Illustration of one of the four cases in \eqref{eq:Ruv}: the maximizing path in $\mathcal{R}_{uv}$ leaves $\mathcal{B}^i$ through the right and enters $\mathcal{C}^j$ through the bottom. We have 
$
L(\mathcal{R}_{uv}) \le L(\mathcal{R}_{u_iv_{j+1}})+L(\mathcal{B}^i)+L(\mathcal{C}^j).
$
}
\label{Fig:Ruv}
\end{figure}
Therefore, for  $0\leq i,j\leq t^{\kappa-\delta}$, we have
\begin{multline}\label{eq:Ruv}
\sup_{\substack{u\in [u_i,u_{i+1}] \\ v\in [v_j,v_{j+1}]}}L(\mathcal{R}_{uv})\le \max(L(\mathcal{R}_{u_iv_j}),L(\mathcal{R}_{u_iv_{j+1}}),L(\mathcal{R}_{u_{i+1}v_j}),L(\mathcal{R}_{u_{i+1}v_{j+1}}))\\
+L(\mathcal{B}^i)+L(\mathcal{C}^j).
\end{multline}
The random variables $L(\mathcal{B}^i)$, $L(\mathcal{C}^j)$ both have the same distribution as $L(t^\delta,\mu t^\delta)$. Since $\delta<1/3-1/32$, using Proposition \ref{Prop:Flu_Ln_continu}, we get
$$\P(L([-u_{i+1},-u_{i}]\times[\mu u_{i},\mu u_{i+1}])\ge t^{1/3-1/32})\le \frac{C}{t}.$$
Thus, an union bound implies Proposition \ref{Prop:flutupara}.
\end{proof}

\subsection{Fluctuations: The non-central case. Proof of Theorem \ref{Th:FluctuLipschitz} (ii).}

\paragraph{Fluctuations are of order $\gg t^{1/3}$.} We will show that, in the non-central case,
 for all $x\in \R$
$$\lim_{t \to \infty}\P\left(
\frac{L^{\alpha,\beta}(at,bt)-t f^{\alpha,\beta}(a,b)}
{ t^{1/3}}\le x\right)=0.
$$
The key argument is to exploit a result by Johansson \cite{JohanssonTransversal} for the \emph{transversal fluctuations} in the Hammersley problem.
We make the assumption that 
$$
\frac{b}{a}< \frac{2}{\tfrac{1}{\alpha}+\tfrac{1}{\beta}}\Leftrightarrow \rho/\sigma <\mu^{-1}
$$
where, as in the proof of Theorem \ref{MaxiTheoreme}, $\rho,\sigma,\mu$ are defined in \eqref{eq:Def_cc'mu},\eqref{eq:Def_rhosigma}.
Let $\mathcal{P}_1=PQRS$ be the parallelogram with $Q=(0,0)$, $S=(\sigma,\rho)$ and such that $P=c(-1,\mu)$
and $R=c'(\mu,-1)$ for some $c,c'>0$.
Recall Lemma \ref{Lem:geometrique_rho}: if $\xi:=\frac{\rho\mu +\sigma}{2\mu}$,
the rectangle $\mathcal{R}:=[0,\mu \xi]\times[0,\xi]$ is a maximizing rectangle inside $\mathcal{P}_1$. Any translation of this rectangle by a vector $u(\mu,-1)$ with $u\in[0,(\sigma-\mu\rho)/2\mu]$ remains a maximizing rectangle.
Denote $\mathcal{R}_t$ a homothetic transformation by a factor $t$ of $\mathcal{R}$ and let 
 
$$
\mathcal{S}_t := \left\{(x,y)\in\mathcal{R}_t, |\mu y-x|\le t^{3/4}\right\}$$
be a " small cylinder" around the diagonal of $\mathcal{R}_t$.
The following Lemma is a weak consequence of  \cite[Th.1.1]{JohanssonTransversal} (take $\gamma=3/4$ in eq.(1.6) in \cite{JohanssonTransversal}).

\begin{figure}
\begin{center}
\includegraphics[width=8cm]{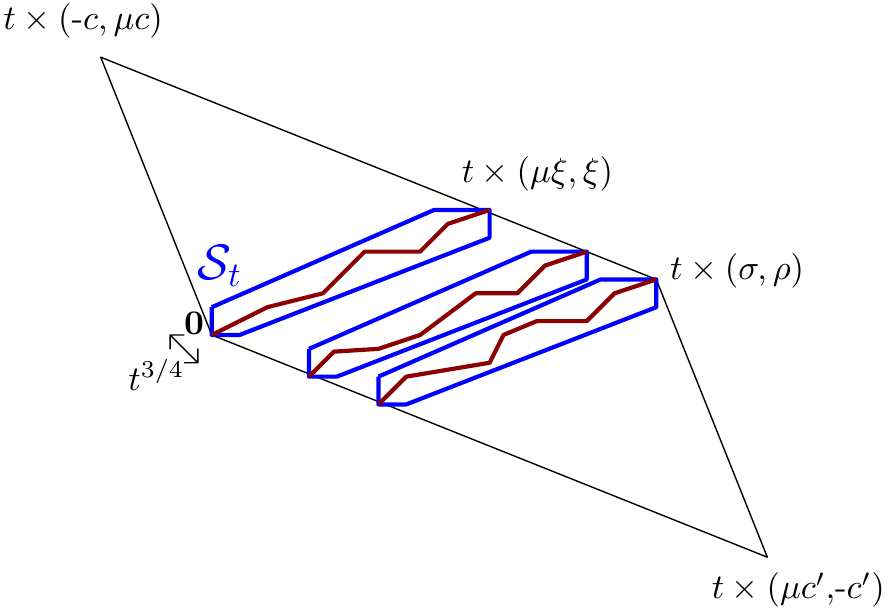}
\caption{Example of a parallelogram with $\rho/\sigma<\mu^{-1}$.}
\label{Fig:Fluctuationnoncentral}
\end{center}
\end{figure}
 
\begin{lem}[Johansson (2000)]
$$\lim_{t \to \infty}\P(L(\mathcal{S}_t)=L(\mathcal{R}_t))=1.$$
\end{lem}
In particular, $$N_t:=\frac{L(\mathcal{S}_t)-2\sqrt{\mu}\xi t}{(\sqrt{\mu}\xi t)^{1/3}}$$ tends to a Tracy-Widom distribution.
Let $
\mathcal{S}_t(u)$ be the image of $\mathcal{S}_t$ after a translation of vector $ut(\mu,-1)$. We clearly have
$$L(\mathcal{P}_t)\ge \sup_{0\le u\le (\sigma-\mu\rho)/2\mu} L(\mathcal{S}_t(u)).$$
Moreover, the random variables $L(\mathcal{S}_t(u))$ and $L(\mathcal{S}_t(u'))$ have the same distribution and are independent if $|u-u'|\ge 2t^{-1/4}$, since, in this case, $
\mathcal{S}_t(u)\cap 
\mathcal{S}_t(u')=\emptyset$.
Hence, we get, with $c_0:=\frac{\sigma-\mu \rho}{4\mu}$, the stochastic lower bound
$$\frac{L(\mathcal{P}_t)-2\sqrt{\mu}\xi t}{(\sqrt{\mu}\xi t)^{1/3}} \succcurlyeq  \max_{1\le i\le c_0t^{1/4}}N^i_t$$
  where $(N^i_t)_{i\ge 1}$ are i.i.d. random variables with common distribution $N_t$. 

Therefore, for every real $x$,
$$
\mathbb{P}\left(\frac{L(\mathcal{P}_t)-2\sqrt{\mu}\xi t}{(\sqrt{\mu}\xi t)^{1/3}} \leq x\right)
\leq \mathbb{P}\left( N^1_t \leq x\right)^{c_0t^{1/4}} \stackrel{t\to +\infty}{\to} 0.
$$

\paragraph{Fluctuations are of order $\ll t^{1/3+\eps}$.}

Here we can re-use the strategy of the proof of the limiting shape. If we take \eqref{eq:Majot6},\eqref{eq:Majot7} with
$$
\gamma t = t f^{\alpha,\beta}(a,b) + At^{1/3+\eps}
$$
for a given real $A$ then we obtain with Proposition \ref{Prop:Flu_Ln_continu} that
$$
\mathbb{P}\left(L(\mathcal{P}_t) \geq t f^{\alpha,\beta}(a,b) +At^{1/3+\eps} \right)
\to 0.
$$

\section*{Concluding remarks}

\begin{itemize}
\item In the first version of this paper we had announced a bit hastily that  our proof of Theorem \ref{Th:FluctuLipschitz} could extend to the critical points  $ \frac{b}{a}=\frac{\alpha+\beta}{2}$ and $ \frac{b}{a} =  \frac{2}{1/\beta+1/\alpha}$. 
It now seems less obvious to us that fluctuations are still of order $t^{1/3}$ at these points.
In any case, our proof does not work in this setting.
\item There is a little room left between $t^{1/3}$ and $t^{1/3+\eps}$ in the non-central case of Theorem \ref{Th:FluctuLipschitz}. We leave open the question of the exact order of fluctuations.

\item Proposition \ref{Prop:OptiLosange} (Hammersley problem in a parallelogram) can be generalized
to an arbitrary convex domain $D$. We would get that
$$
 \lim_{t\to \infty} \frac{L(t\times D)}{t} =
 \sup\{2\sqrt{(v'-v)(u'-u)},\  (u,v) \preccurlyeq (u',v')\in D\}.
$$
Indeed, the proof of the upper bound is easily generalized. Regarding the lower bound, \eqref{eq:MinoRectangleInclus} is not relevant anymore as the whole rectangle $[u,v]\times [u',v']$ may not be  included in $\mathcal{P}_1$. In order to generalize the proof we need to use the localization of the longest path in $[u,v]\times [u',v']$ along the diagonal.
\item Even if we have not written all the details, it seems possible with a little work to deduce from this article that the wandering exponent in the central case is $\chi=2/3$ as expected. 
The idea is that our proof of Proposition  \ref{Prop:flutupara} shows that maximizing paths within the parallelogram start from a point located at most $t^\kappa$ from the origin. As $\kappa < 2/3$ the transversal fluctuations of such paths should be still of order $t^{2/3}$.
\end{itemize}

\subsection*{Aknowledgements}
We would like to thank the referee for her/his very careful reading. In particular it helped us to clarify the agreement between our asymptotics and the space-time scaling of the KPZ universality class.

\end{document}